\documentclass[twoside]{article}

\usepackage[accepted]{aistats2022}
\usepackage[round]{natbib}

\usepackage{wrapfig}

\usepackage[utf8]{inputenc} 
\usepackage[T1]{fontenc}    
\usepackage{url}            
\usepackage{booktabs}       
\usepackage{amsfonts}       
\usepackage{nicefrac}       
\usepackage{microtype}      
\usepackage{nccmath}

\usepackage{amsthm}

\usepackage{thm-restate}

\usepackage{color, xcolor}
\definecolor{mydarkblue}{rgb}{0,0.08,0.45}
\usepackage[colorlinks=true,
    linkcolor=mydarkblue,
    citecolor=mydarkblue,
    filecolor=mydarkblue,
    urlcolor=mydarkblue,
    pdfview=FitH]{hyperref}

\usepackage{cuted}
\usepackage{amsmath}
\usepackage{amssymb}
\usepackage{amsfonts}
\usepackage{tikz}
\usepackage{enumitem}
\usepackage{float}
\usepackage{stmaryrd}
\usepackage[ruled,vlined]{algorithm2e}
\usepackage{tocloft}

\usepackage{graphicx}

\theoremstyle{plain}

\newtheorem{Th}{Theorem}[section]

\newtheorem{Prop}[Th]{Proposition}

\theoremstyle{definition}
\newtheorem{Def}[Th]{Definition}

\newtheorem{Rem}[Th]{Remark}
\newtheorem{?}[Th]{Problem}
\newtheorem{Ex}[Th]{Example}

\usepackage{cleveref} 
\usepackage{bm}

\crefname{Th}{Theorem}{Theorems}

\crefname{lemma}{Lemma}{Lemmas}
\crefname{fact}{Fact}{Facts}
\crefname{theorem}{Theorem}{Theorems}
\crefname{corollary}{Corollary}{Corollaries}
\crefname{Prop}{Proposition}{Propositions}

\crefname{claim}{Claim}{Claims}
\crefname{example}{Example}{Examples}
\crefname{problem}{Problem}{Problems}
\crefname{definition}{Definition}{Definitions}
\crefname{assumption}{Assumption}{Assumptions}
\crefname{subsection}{Subsection}{Subsections}
\crefname{section}{Section}{Sections}
\crefname{algorithm}{Algorithm}{Algorithms}
\crefname{algocf}{alg.}{algs.}
\Crefname{algocf}{Algorithm}{Algorithms}
\crefname{proposition}{Proposition}{Propositions}
\crefname{exemple}{Exemple}{Examples}
\crefname{remark}{Remark}{Remarks}

\usepackage{xargs}
\definecolor{brickred}{rgb}{0.8, 0.25, 0.33}

\newcommandx{\aymeric}[1]{\textcolor{brickred}{\textbf{AD: #1}}}
\newcommandx{\baptiste}[1]{\textcolor{green}{\textbf{BG: #1}}}
\newcommandx{\adrien}[1]{\textcolor{blue}{\textbf{AT: #1}}}
\newcommandx{\fabian}[1]{\textcolor{magenta}{\textbf{FP: #1}}}
\newcommandx{\damien}[1]{\textcolor{purple}{\textbf{DS: #1}}}

\DeclareMathOperator*{\argmax}{arg\,max}
\DeclareMathOperator*{\argmin}{arg\,min}
\DeclareMathOperator*{\spanop}{{span}}

\newcommand{\tausigmaone}{{\textcolor[HTML]{66C2A5}{$\bm{\tau^{\sigma_1^{\Lambda}}}$}}}
\newcommand{\tausigmatwo}{\textcolor[HTML]{FC8D62}{$\bm{\tau^{\sigma_2^{\Lambda}}}$}}
\newcommand{\tausigmathree}{\textcolor[HTML]{984EA3}{$\bm{\tau^{\sigma_3^{\Lambda}}}$}}

\newcommand{\muone}{\textcolor[HTML]{984ea3}{\boldsymbol{\mu}_1}}
\newcommand{\Lone}{\textcolor[HTML]{4DAF4A}{\boldsymbol{L}_1}}
\newcommand{\mutwo}{\textcolor[HTML]{4DAF4A}{\boldsymbol{\mu}_2}}
\newcommand{\Ltwo}{\textcolor[HTML]{984ea3}{\boldsymbol{L}_2}}

\usepackage{fancyhdr} 
\begin{document}
\addtocontents{toc}{\protect\setcounter{tocdepth}{0}}

%

%
\runningauthor{Goujaud, Scieur, Dieuleveut, Taylor, Pedregosa}

\twocolumn[

\aistatstitle{Super-Acceleration with Cyclical Step-sizes}

\aistatsauthor{Baptiste Goujaud \\ CMAP, École Polytechnique \\
    Institut Polytechnique de Paris \And Damien Scieur \\ Samsung SAIL Montreal 
\And  Aymeric Dieuleveut \\ CMAP, École Polytechnique \\
    Institut Polytechnique de Paris \AND  Adrien Taylor \\INRIA, \'Ecole Normale Supérieure \\ CNRS, PSL Research University, Paris\\ \\
\And Fabian Pedregosa \\ Google Research}

]

\begin{abstract}
We develop a convergence-rate analysis of momentum with cyclical step-sizes. We show that under some assumption on the spectral gap of Hessians in machine learning, cyclical step-sizes are provably faster than constant step-sizes. More precisely, we develop a convergence rate analysis for quadratic objectives that provides optimal parameters and shows that cyclical learning rates can improve upon traditional lower complexity bounds. We further propose a systematic approach to design optimal first order methods for quadratic minimization with a given spectral structure. Finally, we provide a local convergence rate analysis beyond quadratic minimization for the proposed methods and illustrate our findings through benchmarks on least squares and logistic regression problems.
\end{abstract}



\section{Introduction}\label{sec:introduction-and-related-work}

One of the most iconic methods in first order optimization is gradient descent with momentum, also known as the heavy ball method~\citep{polyak1964some}. 
This method enjoys widespread popularity both in its original formulation and in a stochastic variant that replaces the gradient by a stochastic estimate, a method that is behind many of the recent breakthroughs in deep learning~\citep{sutskever2013importance}.

A variant of the stochastic heavy ball where the step-sizes are chosen in \emph{cyclical} order has recently come to the forefront of machine learning research, showing state-of-the art results on different deep learning benchmarks  \citep{loshchilov2016sgdr,smith2017cyclical}. Inspired by this empirical success, we aim to study the convergence of the heavy ball algorithm where step-sizes $h_0, h_1, \ldots$ are not fixed or decreasing but instead chosen in cyclical order, as in Algorithm \ref{algo:cyclical_heavy_ball}.

\begin{restatable}[t]{algorithm}{HBK}
\caption{\newline Cyclical heavy ball $\mathrm{HB}_K(h_0,\ldots, h_{K-1}; m)$\label{algo:cyclical_heavy_ball}}
\SetAlgoLined
\textbf{Input:} Initialization $x_0$, momentum $m \in (0, 1)$, step-sizes $\{h_0,\ldots, h_{K-1}\}$\\
$x_1 = x_0 - \mfrac{h_0}{1 + m}\nabla f(x_0)$\\
\For{$t = 1,\, 2,\,\ldots$}{
\begin{align*}
    x_{t+1} = & ~~ x_t - h_{\text{mod}(t,K)} \nabla f(x_t) + m(x_{t}-x_{t-1})
\end{align*}
 \vspace{-4ex}
}
\end{restatable}

The heavy ball method with constant step-sizes enjoys a mature theory, where it is known for example to achieve optimal black-box worst-case complexity of quadratic convex optimization~\citep{nemirovsky1992information}. In stark contrast, little is known about the the convergence of the above variant with cyclical step-sizes. Our main motivating question is

\begin{center}
\fbox{\parbox{0.45\textwidth}{
\begin{center}
Do cyclical step-sizes improve \\ convergence of heavy ball?
\end{center}
}}
\end{center}

Our {\bfseries main contribution} provides a positive answer to this question and, more importantly,  \emph{quantifies} the speedup under different assumptions. In particular, we show that for quadratic problems, whenever Hessian's spectrum belongs to two or more disjoint intervals, the heavy ball method with cyclical step-sizes achieves a faster worst-case convergence rate. Recent works have shown that this assumption on the spectrum is quite natural and occurs in many machine learning problems, including deep neural networks \citep{sagun2017empirical, papyan2018full, ghorbani2019investigation, papyan2019measurements}. The concurrent work of \citet{oymak2021super} analyzes gradient descent (without momentum, see extended comparison in Appendix \ref{apx:comparison_with_Oymak}) under this assumption.
More precisely, we list our main contributions below.
\begin{itemize}[leftmargin=*]
    \item In sections \ref{sec:alternating_hb} and \ref{sec:polynomials}, we provide a {\bfseries tight convergence rate analysis} of the cyclical heavy ball method (Theorems \ref{thm:rate_factor_alternating_hb} and~\ref{cor:rate_convergence_heavy_ball_alternating_step_size} for two step-sizes, and Theorem \ref{thm:general_rate_convergence} for the general case). This analysis highlights a regime under which this method achieves a faster worst-case rate than the accelerated rate of heavy ball, a phenomenon we refer to as \emph{super-acceleration}. Theorem \ref{thm:local_convergence_non_quadratic} extends the (local) convergence rate analysis results to non-quadratic objectives.

    \item As a byproduct of the convergence-rate analysis, we obtain an  explicit expression for the {\bfseries optimal parameters} in in the case of cycles of length two (Algorithm \ref{algo:alternating_heavy_ball}) and an implicit expression in terms of a system of $K$ equations in the general case.
    
    \item Section \ref{sec:experiments} presents {\bfseries numerical benchmarks} illustrating the improved convergence of the cyclical approach on 4 problems involving quadratic and logistic losses on both synthetic and a handwritten digits recognition dataset.
    \item Finally, we conclude in  Section \ref{sec:conclusions} with a discussion of this work's {\bfseries limitations}.
\end{itemize}

\vspace{-0.4em}
\section{Notation and Problem Setting}\label{sec:notations}
\vspace{-0.8em}
Throughout the paper (except in Section~\ref{sec:local-convergence}), we consider the problem of minimizing a quadratic function:
\begin{equation}
    \min_{x \in \mathbb{R}^d} f(x)\,, ~ \text{with } \;\; f\in \mathcal{C}_\Lambda, \label{eq:quad_problem}\tag{OPT}
\end{equation}

where $\mathcal{C}_\Lambda$ is the class of quadratic functions with Hessian matrix $H$ and whose Hessian spectrum $\mathrm{Sp}(H)$ is localized in $\Lambda \subseteq [\mu,L] \subseteq \mathbb{R}_{> 0}$:

\begin{equation*}
    \mathcal{C}_\Lambda \triangleq \left\{f(x) = (x-x_*)^{\top}\tfrac{H}{2}(x-x_*) + f_*,\, \mathrm{Sp}(H)\subseteq \Lambda \right\} \,
\end{equation*}

The condition $\Lambda \subseteq [\mu,L]$ implies all quadratic functions under consideration are $L$-smooth and $\mu$-strongly convex. For this function class, we define $\kappa$, the (inverse) condition number, and  $\rho$, the ratio between the center of $\Lambda$ and its radius, as
\begin{equation}\label{eq:kappa_rho}
 \kappa \triangleq \frac{\mu}{L}, \qquad \quad \rho \triangleq \frac{L+\mu}{L-\mu} \quad\textcolor{gray}{=\left(\frac{1+\kappa}{1-\kappa}\right)}\,.
\end{equation}

Finally, for a method solving~\eqref{eq:quad_problem} that generates a sequence of iterates $\{x_t\}$, we define its worst-case rate $r_t$ and its asymptotic rate factor $\tau$ as
\begin{equation} \label{eq:def_rates}
    r_t \triangleq \underset{x_0\in\mathbb{R}^d,\; f\in\mathcal{C}_{\Lambda}}{\sup}  \frac{\|x_t-x_*\|}{\|x_0-x_*\|}  ,\; 1-\tau \triangleq \limsup\limits_{t\rightarrow\infty} \sqrt[t]{r_t}\,.
\end{equation}

\section{Super-acceleration with Cyclical Step-sizes} \label{sec:alternating_hb}
\vspace{-0.8em}

In this section we develop one of our main contributions, a convergence rate analysis of the cyclical heavy ball method with cycles of length 2.
This analysis crucially depends on the location of the Hessian's eigenvalues; we assume that these are contained in a set $\Lambda$ that is the union of 2 intervals \textit{of the same size}
\begin{equation} \label{eq:def_lambda}
    \vphantom{\sum_i^n}\Lambda = [\muone,\,\Lone]\cup[\mutwo,\,\Ltwo]\,,~ \Lone-\muone = \Ltwo-\mutwo\,.
\end{equation}
By symmetry, this set is alternatively described by
\begin{equation}
\vphantom{\sum_i^n}
    \mu \triangleq \muone, \quad L \triangleq \Ltwo \quad \text{and} \quad R \triangleq \frac{\textcolor[HTML]{4DAF4A}{\boldsymbol{\mu}_2 - \boldsymbol{L}_1}}{\textcolor[HTML]{984ea3}{\boldsymbol{L}_2 - \boldsymbol{\mu}_1}}\,, \label{eq:def_params_hb_alternating}
\end{equation} 
where $R$ is the relative length of the gap $\textcolor[HTML]{4DAF4A}{\boldsymbol{\mu}_2 - \boldsymbol{L}_1}$ with respect to the diameter $\textcolor[HTML]{984ea3}{\boldsymbol{L}_2 - \boldsymbol{\mu}_1}$ (see Figure \ref{fig:mnist_density}). This parametrization is convenient since the relative gap plays a crucial role in our convergence analysis. Our results allow $R=0$, therefore recovering the classical setting of Hessian eigenvalues contained in an interval. 

\begin{figure}[ht]
    \centering
    \includegraphics[width=0.7\linewidth]{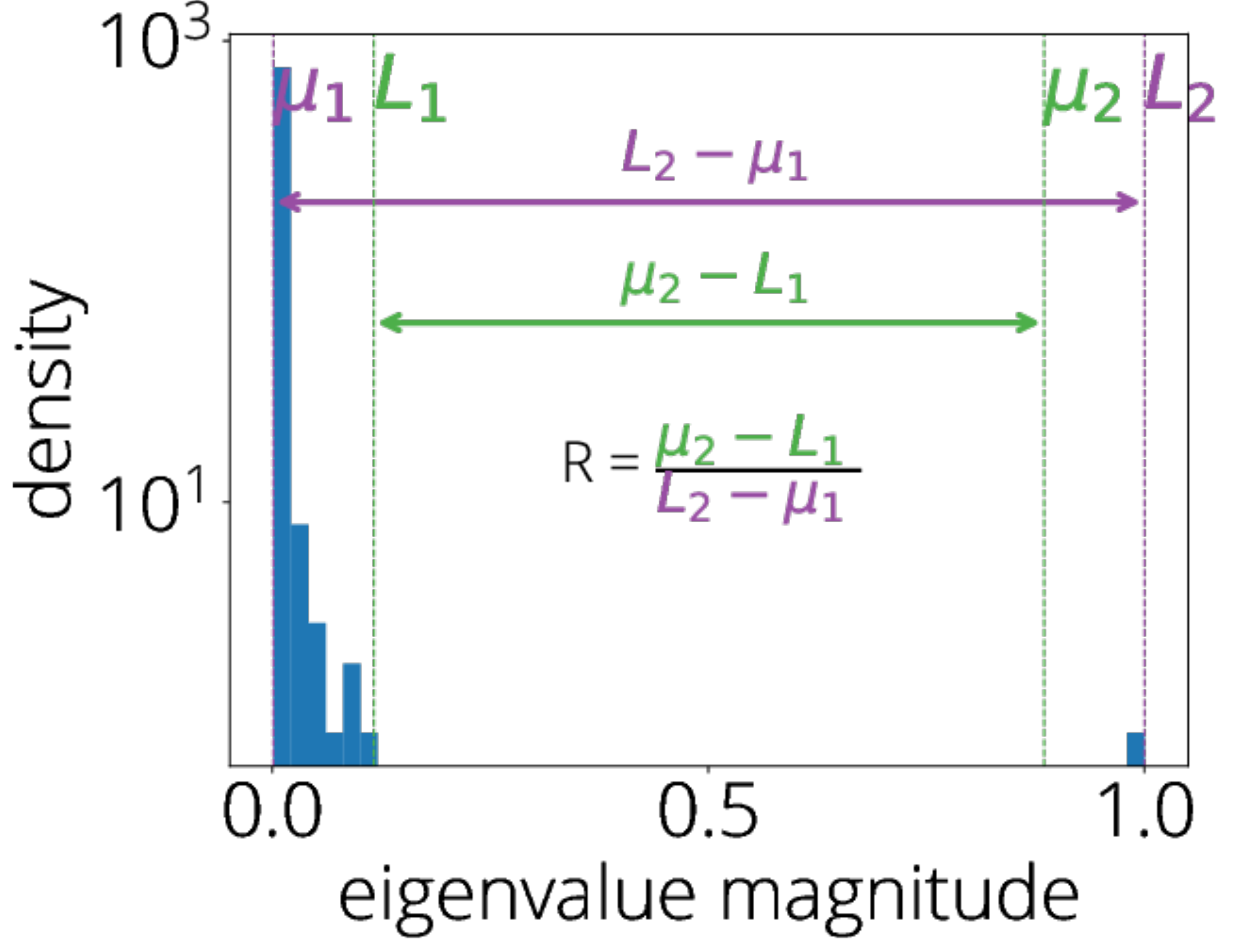}
    \caption{Hessian eigenvalue histogram  for a  quadratic objective on MNIST. The outlier eigenvalue at ${\textcolor[HTML]{984ea3}{\boldsymbol{L}_2}}$ generates a non-zero relative gap $R=0.77$. In this case, 
    the 2-cycle heavy ball method has a faster asymptotic rate than the single-cycle one (see Section \ref{subsec:comparison_polyak}).}
    \label{fig:mnist_density}
\end{figure}


Through a correspondence between optimization methods and polynomials (see Section \ref{sec:polynomials}), we can derive a worst-case analysis for the cyclical heavy ball method. The outcome of this analysis is in the following theorem, that provides the asymptotic convergence rate of Algorithm \ref{algo:cyclical_heavy_ball} for cycles of length two. All proofs of results in this section can be found in Appendix \ref{apx:case-of-a-2-step-sizes-cycle}.

\begin{figure*}
    \centering
    \includegraphics[width=0.9\linewidth]{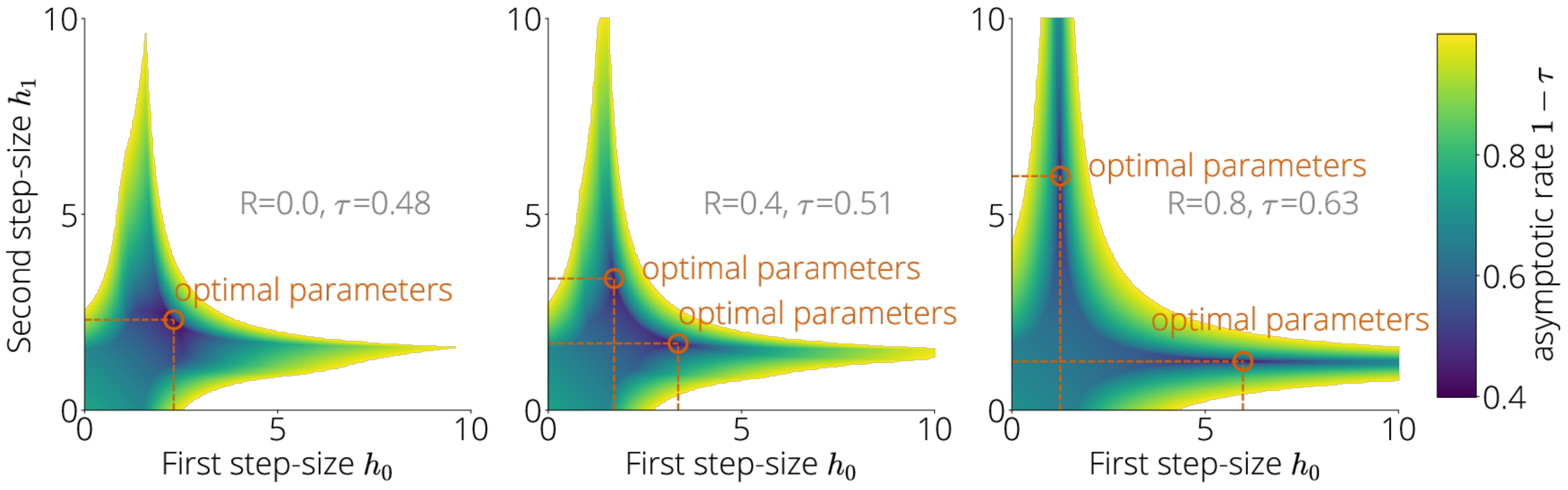}
    \vspace{-1ex}
    \caption{{\bfseries Asymptotic rate of cyclical ($K=2$) heavy ball} in terms of its step-sizes $h_0, h_1$ across 3 different values of the relative gap $R$. In the {\bfseries left} plot, the relative gap is zero, and so the step-sizes with smallest rate coincide ($h_0 = h_1$). For non-zero values of $R$ ({\bfseries center and right}), the optimal method instead alternates between two \emph{different} step-sizes.
    In all plots the momentum parameter $m$ is set according to Algorithm \ref{algo:alternating_heavy_ball}. 
    }
    \label{fig:rate_factor_cyclical}
\end{figure*}

\begin{restatable}[Rate factor of $\mathrm{HB}_2(h_0, h_1; m)$]{Th}{ratefactoralternatinghb}
 \label{thm:rate_factor_alternating_hb}
    Let $f\in\mathcal{C}_\Lambda$ and consider the cyclical heavy ball method with step-sizes $h_0,\,h_1$ and momentum parameter $m$. The asymptotic rate factor of Algorithm~\ref{algo:cyclical_heavy_ball} with cycles of length two is
    \begin{equation*}
        1-\tau =
        \left\{
            \begin{array}{ll}
                 \sqrt{m} & \mathrm{if~} \sigma_{*} \leq 1, \\
                 \sqrt{m} \left( \sigma_{*} + \sqrt{\sigma_{*}^2 - 1} \right)^{\frac{1}{2}} & \mathrm{if~} \sigma_{*} \in \Big(1,\; \frac{1+m^2}{2m}\Big), \\
                 \geq 1 ~ \text{(no convergence)} & \mathrm{if~} \sigma_{*} \geq \frac{1+m^2}{2m},
            \end{array}
        \right.\,
    \end{equation*}
    \begin{equation*}
        \text{ with}\quad \sigma_{*} = \underset{\lambda\in\left\{ \mu_1, L_1, \mu_2, L_2, (1+m)\frac{h_0 + h_1}{2 h_0 h_1} \right\}\cap \Lambda}{\max}\left| \sigma_2(\lambda) \right|\,
    \end{equation*}
    \begin{equation*}
        \text{ and}\quad \sigma_2(\lambda) = 2\left(\mfrac{1+m - \lambda h_0}{2\sqrt{m}}\right)\left(\mfrac{1+m - \lambda h_1}{2\sqrt{m}}\right) - 1 \,.
    \end{equation*}
\end{restatable}

\subsection{Optimal algorithm}

The previous theorem gives the convergence rate for all triplets ($h_0, h_1, m$). This allows us for instance to map out the associated convergence rate for every pair of step-sizes. As we illustrate in Figure \ref{fig:rate_factor_cyclical}, as we increase the relative gap ($R$), the optimal step-sizes become further apart.

Another application of the previous theorem is to find the parameters that minimize the asymptotic convergence rate. Although the process rather tedious and relegated to Appendix  \ref{apx:case-of-a-2-step-sizes-cycle}, the resulting momentum $(m)$ and step-size parameters $(h_0, h_1)$ are remarkably simple, and given by the expressions
\begin{align}
     &m = \left(\frac{\sqrt{\rho^2 - R^2} - \sqrt{\rho^2 - 1}}{\sqrt{1-R^2}}\right)^2 \\
        &h_0 =  \frac{1+m}{\Lone}\,\quad ~
        h_1 = \frac{1+m}{\mutwo}\,.
\end{align}

Being one of our main contributions, this algorithm is also described in pseudocode in Algorithm \ref{algo:alternating_heavy_ball}.
By construction, this method has an \emph{asymptotically optimal} convergence rate which we detail in the next Corollary:

\begin{algorithm}
    \caption{Cyclical ($K=2$) heavy ball with optimal parameters \label{algo:alternating_heavy_ball}}
    \SetAlgoLined
    \textbf{Input:} Initial iterate $x_0$, $\muone < \Lone \leq \mutwo < \Ltwo$ $\quad$ (where $\Lone-\muone = \Ltwo-\mutwo$)\\
    \textbf{Set:} $\rho = \frac{\Ltwo+\muone}{\Ltwo-\muone}$, $R=\frac{\textcolor[HTML]{4DAF4A}{\boldsymbol{\mu}_2 - \boldsymbol{L}_1}}{\textcolor[HTML]{984ea3}{\boldsymbol{L}_2 - \boldsymbol{\mu}_1}}$,\\
    $m =  \left(\frac{\sqrt{\rho^2 - R^2} - \sqrt{\rho^2 - 1}}{\sqrt{1-R^2}}\right)^2$ \\
    $x_1 = x_0 - \frac{1}{\Lone}\nabla f(x_0)$ \\
    ~\\
    \For{$t = 1,\, 2,\,\ldots$}{
        \vspace{1ex}
        \begin{gather*}
            h_t = \textstyle \frac{1+m}{\Lone} ~ \text{(if $t$ is even)}, \quad h_t = \frac{1+m}{\mutwo} ~ \text{(if $t$ is odd)} \\
            x_{t+1} = x_t - h_t \nabla f(x_t) + m(x_{t}-x_{t-1})
        \end{gather*}
        \vspace{-4ex}
    }
\end{algorithm}

\begin{restatable}{Cor}{rateconvergenceheavyballalternatingstepsize} \label{cor:rate_convergence_heavy_ball_alternating_step_size}
    The non-asymptotic and asymptotic worst-case rates $r_t^{\text{Alg. 2}}$ and $1-
    \tau^{\text{Alg. 2}}$ of Algorithm~\ref{algo:alternating_heavy_ball} over $\mathcal{C}_\Lambda$ for even iteration number $t$ are
    \begin{align*}
        &r_{t}^{\text{Alg. 2}} = \textstyle\Big(\frac{\sqrt{\rho^2 - R^2} - \sqrt{\rho^2 - 1}}{\sqrt{1 - R^2}}\Big)^{t} \Big( 1 + t\sqrt{\frac{\rho^2-1}{\rho^2 - R^2}}\Big),
        \\
        &1 - \tau^{\text{Alg. 2}} = \mfrac{\sqrt{\rho^2 - R^2} - \sqrt{\rho^2 - 1}}{\sqrt{1 - R^2}}\,.
    \end{align*}
\end{restatable}

Note that this result also holds if we swap the 2 step-sizes in Algorithm \ref{algo:alternating_heavy_ball}.

\paragraph{Eigengap and accelerated cyclical step-sizes}

While Corollary \ref{cor:rate_convergence_heavy_ball_alternating_step_size} focuses on the optimal tuning of \Cref{algo:alternating_heavy_ball}, \Cref{thm:cyclical_chebyshev_2_steps} provides general convergence analysis for non-optimal parameters. In the case of existence of an eigengap, a range of cyclical step-sizes leads to an accelerated rate of convergence (compared to the optimal constant step-size strategy) and therefore, an inexact parameters search can lead to such an acceleration.

\subsection{Comparison with Polyak Heavy Ball}\label{subsec:comparison_polyak}

    In the absence of eigenvalue gap ($R=0$ and $\Lambda = [\mu, L]$), Algorithm~\ref{algo:alternating_heavy_ball} reduces to Polyak heavy ball~(PHB) \citep{polyak1964some}, whose worst-case rate is detailed in Appendix \ref{apx:optimal-methods-for-strongly-convex-smooth-quadratic-objective}. Since the asymptotic rate of Algorithm ~\ref{algo:alternating_heavy_ball} is monotonically decreasing in $R$, the convergence rate of the cyclical variant is always better than PHB. Furthermore, in the ill-conditioned regime (small $\kappa$), the comparison is particularly simple: the optimal 2-cycle algorithm has a  $\sqrt{1-R^2}$ relative improvement over PHB, as provided by the next proposition. A~more thorough comparison for different support sets $\Lambda$ is discussed in \Cref{tab:speedups}.
    
    \begin{restatable}{Prop}{asymptoticratesmallkappa}
        \label{prop:asymptotic_rate_small_kappa}
        Let $ R \in [0,1)$. The rate factors of respectively Algorithm~\ref{algo:alternating_heavy_ball} and PHB verify
        \begin{gather}
            1-\tau^{\text{Alg. 2}} \underset{\kappa\rightarrow 0}{=} \textstyle 1 - \frac{2\sqrt{\kappa}}{\sqrt{1-R^2}} + {o}(\sqrt{\kappa})\,,\\
            1-\tau^{\text{PHB}} \underset{\kappa\rightarrow 0}{=} 1 - 2\sqrt{\kappa} + {o}(\sqrt{\kappa})\,. \nonumber
        \end{gather}
    \end{restatable}

\begin{table*}[ht]
{
\begin{center}
{\renewcommand{\arraystretch}{1.8}
\begin{tabular}{@{}llll@{}}
\specialrule{2pt}{1pt}{1pt}
Relative gap $R\;\;$ & Set $\Lambda$ & Rate factor $\tau$ & Speedup $\tau/\tau^{\text{PHB}}$ \\
\specialrule{2pt}{1pt}{1pt}
$R\in[0, 1)$ & $[\mu, \mu + \frac{1-R}{2}(L-\mu)]\cup[L - \frac{1-R}{2}(L-\mu), L]$ &  $\frac{2\sqrt{\kappa}}{\sqrt{1-R^2}}$& $(1-R^2)^{-\frac{1}{2}}$ \\ \hline
$R=1 - \sqrt{\kappa}/2$  & $[\mu, \mu + \frac{\sqrt{\mu L}}{4}] \cup [L - \frac{\sqrt{\mu L}}{4}, L]$  & $2\sqrt[4]{\kappa}$ & $\kappa^{-\frac{1}{4}}$  \\ \hline
$R=1 - 2\gamma\kappa$ & $[\mu, (1+\gamma)\mu] \cup [L - \gamma \mu, L]$ &  indep. of $\kappa$ &  $O(\kappa^{-\frac{1}{2}})$\\ 
\specialrule{2pt}{1pt}{1pt}\vspace{0em}
\end{tabular}
\caption{Case study of the convergence of Algorithm \ref{algo:alternating_heavy_ball} as a function of $R$, in the regime $\kappa \rightarrow 0$. The \textbf{first line} corresponds to the regime where $R$ is independent of $\kappa$, and we observe a constant gain w.r.t. PHB. The {\bfseries second line} considers a setting in which $R$ depends on $\sqrt{\kappa}$, that is, the two intervals in $\Lambda$ are relatively small. The asymptotic rate reads $(1-2\sqrt[4]{\kappa})^t$, improving over the $(1-2\sqrt{\kappa})^t$ rate of Polyak Heavy ball, unimprovable when $R=0$.
Finally, in the \textbf{third line}, $R$ depends on $\kappa$, the two intervals in $\Lambda$ are so small that the convergence becomes $O(1)$, i.e., is independent of $\kappa$.\vspace{-2ex}}\label{tab:speedups}}
\end{center}}
\end{table*}

\section{A constructive Approach: Minimax Polynomials}
\label{sec:polynomials}

This section presents a generic framework that allows designing optimal momentum and step-size cycles for given sets $\Lambda$ and cycle length $K$.

We first recall classical results that link optimal first order methods on quadratics and Chebyshev polynomials. Then, we generalize the approach by showing that optimal methods can be viewed as combinations of Chebyshev polynomials, and minimax polynomials $\sigma_K^{\Lambda}$ of degree $K$ over the set $\Lambda$. Finally, we show how to recover the step-size schedule from $\sigma_K^{\Lambda}$ and present the general algorithm (Algorithm \ref{algo:cyclical_chebyshev}).

\subsection{First Order Methods on Quadratics and Polynomials}\label{sec:lower-bounds-and-optimal-methods-for-smooth-and-strongly-convex-quadratic-objective-functions}

A key property that we will use extensively in the analysis is the following link between first order methods and polynomials.

\begin{restatable}{Prop}{linkalgopoly}\label{prop:link_algo_poly}
    Let $f\in\mathcal{C}_{\Lambda}$. The iterates $x_t$ satisfy
    \begin{equation}
        x_{t+1} \in x_0 + \spanop\{ \nabla f(x_0),\ldots, \nabla f(x_t) \} \,, \label{def:first_order_algo}
    \end{equation}
    where $x_0$ is the initial approximation of $x_*$, if and only if there exists a sequence of polynomials $(P_t)_{t\in\mathbb{N}}$, each of degree at most 1 more than the highest degree of all previous polynomials and $P_0$ of degree 0 (hence the degree of $P_t$ is at most $t$), such that
    \begin{equation}
        \vphantom{\sum_i}\forall \;t \quad x_t - x_* = P_t(H)(x_0-x_*), \quad P_t(0)=1\,. \label{eq:link_polynomial}
    \end{equation}
\end{restatable}

\begin{Ex}[Gradient descent]
    Consider the gradient descent algorithm with fixed step-size $h$, applied to problem \eqref{eq:quad_problem}. Then, after unrolling the update, we have
    \begin{align}
        x_{t+1} \textcolor{gray}{-x_*} & = x_{t} \textcolor{gray}{-x_*} - h\nabla f(x_t) \nonumber \\
        & = x_{t} \textcolor{gray}{-x_*} - h H (x_t-x_*) \nonumber \\
        & = (I-hH)^{t+1}(x_0-x_*)\,. \label{eq:example_gd}
    \end{align}
    In this case, the polynomial associated to gradient descent is $P_t(\lambda) = (1-h\lambda)^t$.
\end{Ex}

The above proposition can be used to obtain worst-case rates for first order methods by bounding their associated polynomials. Indeed, using the Cauchy-Schwartz inequality in \eqref{eq:link_polynomial} leads to
\begin{gather}
    \|x_t-x_*\| \leq \sup_{\lambda \in \Lambda} |P_t(\lambda)|\; \|x_0-x_*\|  \nonumber  \\ \Longrightarrow \quad r_t = \sup_{\lambda \in \Lambda} |P_t(\lambda)|,\quad
    \text{where } P_t(0)=1\,.\label{eq:rate_convergence_polynomial}
\end{gather}
Therefore, finding the algorithm with the fastest worst-case rate can be equivalently framed as the problem of finding the polynomial with smallest value on the eigenvalue support $\Lambda$, subject to the normalization condition $P_t(0)=1$. Such polynomials are referred to as \textbf{minimax}. Throughout the paper, we use this polynomial-based approach to find methods with optimal rates.

An important property of minimax polynomials is their \textit{equioscillation} on $\Lambda$ (see Theorem \ref{th:equioscillation} and its proof for a formal statement).
\begin{Def} \label{def:equiosc}
(Equioscillation) A polynomial $P_t$ of degree $t$ equioscillates on $\Lambda$ if it verifies $P_t(0)=1$ and there exist $\lambda_0<\lambda_1<\ldots<\lambda_t \in \overline{\Lambda}$ such that
\begin{equation} \label{def:equioscillation}
    \vphantom{\sum_i}P_t(\lambda_i) = (-1)^i\,\max_{\lambda\in\Lambda}|P_t(\Lambda)|\,.
\end{equation}
\end{Def}

\begin{Ex}[$\Lambda$ is an interval] \label{ex:lambda_interval}
The $t$-th order Chebyshev polynomials of the first kind  $T_t$ satisfy the \emph{equioscillation} property on $[-1,\,1]$. It follows that minimax polynomials on $\Lambda=[\mu,L]$ can be obtained by composing the Chebyshev polynomial $T_t$ with the linear transformation $\sigma_1^{\Lambda}$:
\begin{gather}
    \frac{T_t\left( \sigma_1^{\Lambda}(\lambda) \right)}{T_t\left(\sigma_1^{\Lambda}(0)\right)}= \underset{\substack{
         P \in \mathbb{R}_t[X],
         P(0) = 1
    }}{\argmin} \;\; \underset{\lambda \in \Lambda}{\sup} |P(\lambda)| \,, \label{eq:optimal_poly_smooth_strong_convex} \\
    \text{with } \,\, \sigma_1^{\Lambda}(\lambda) = \frac{L+\mu}{L-\mu} - \frac{2}{L-\mu}\lambda\,, \nonumber
\end{gather}
where $\sigma_1^{\Lambda}$ maps the interval $[\mu,\,L]$ to $[-1,1]$. The optimization method associated with this minimax polynomial is the Chebyshev semi-terative method \citep{flanders1950numerical, golub1961chebyshev}, described also in Appendix \ref{subsec:chebyshev-method}. This method  achieves the lower complexity bound for smooth strongly convex quadratic minimization \citep[Chapter 12]{nemirovski1995information} or \citep{nemirovsky1992information,Nest03a}.
\end{Ex}

The next proposition provides the main results in this subsection, which is key for obtaining Algorithm~\ref{algo:alternating_heavy_ball}. It characterizes the even degree minimax polynomial in the setting of Section~\ref{sec:alternating_hb}, that is, when $\Lambda$ is the union of 2 intervals of same size. In this case, the minimax solution is also based on Chebyshev polynomials, but composed with a degree-two polynomial $\sigma_2^{\Lambda}$.

\begin{restatable}{Prop}{lambdauniontwointerval} \label{prop:lambda_union_two_interval}
    Let $\Lambda = [\muone,\,\Lone]\cup[\mutwo,\,\Ltwo]$ be an union of two intervals of the same size ($\Lone-\muone = \Ltwo-\mutwo$) and let $m, h_0, h_1$ be as defined in Algorithm \ref{algo:alternating_heavy_ball}. Then the minimax polynomial (solution to~\eqref{def:equioscillation}) is, for all $t=2n, \, n\in\mathbb{N}_0^+$,
    \begin{gather*}
        \frac{T_n\left( \sigma_2^{\Lambda}(\lambda) \right)}{T_n\left(\sigma_2^{\Lambda}(0)\right)} = \underset{\substack{
             P \in \mathbb{R}_t[X],\\
             P(0) = 1
        }}{\argmin} \;\; \underset{\lambda \in \Lambda}{\sup} |P(\lambda)| \,, \\
        \text{with } \,\, \sigma_2^{\Lambda}(\lambda) = \frac{1}{2m}\left(1 + m - \lambda h_0\right)\left(1 + m - \lambda h_1\right) - 1 \,.\label{eq:optimal_poly_two_intervals}
    \end{gather*}
\end{restatable}

\subsection{Generalization to Longer Cycles} \label{ssec:finding_sigma}

The polynomial in Example \ref{ex:lambda_interval} uses a linear link function $\sigma_1^\Lambda$ to map $\Lambda$ to $[-1, 1]$. In Proposition~\ref{prop:lambda_union_two_interval}, we see that a degree \emph{two} link function $\sigma_2^{\Lambda}$ can be used to find the minimax polynomial when $\Lambda$ is the union of two intervals. This section generalizes this approach and considers higher-order polynomials for $\sigma_K$.

We start with the following parametrization, with an arbitrary polynomial $\sigma_K$ of degree $K$,
\begin{equation} \label{eq:parametrization_z}
    P_{t}(\lambda; \sigma_K) \triangleq \frac{T_n\left( \sigma_K(\lambda) \right)}{T_n\left(\sigma_K(0)\right)}, \quad \forall t=Kn,\, n\in\mathbb{N}^+_0\,.
\end{equation}

As we will see in the next subsection, this parametrization allows considering cycles of step-sizes. Our goal now is to find the $\sigma_K$ that obtains the fastest convergence rate possible. The next proposition quantifies its impact on the asymptotic rate and its proof can be found in~\Cref{apx:derivation-of-optimal-algorithm-withstep-sizes}.

\begin{restatable}{Prop}{rateconvergencesigma} \label{prop:rate_convergence_sigma}
    For a given $\sigma_K$ such that ${\sup}_{\lambda \in \Lambda} |\sigma_K(\lambda)| = 1$, the asymptotic rate factor $\tau^{\sigma_K}$ of the method associated to the polynomial \eqref{eq:parametrization_z} is
    \begin{gather} 
        \textstyle 1-\tau^{\sigma_K} = \lim\limits_{t\rightarrow\infty} \sqrt[t]{\underset{\lambda \in \Lambda}{\sup} |P_{t}(\lambda; \sigma_K)|} = \left(\sigma_0-\sqrt{\sigma_0^2-1}\right)^{\frac{1}{K}}, \nonumber\\
        \text{with } \sigma_0 \triangleq \sigma_K(0)\,. \label{eq:asymptotic_rate}
    \end{gather}
\end{restatable}

For a fixed $K$, the asymptotic rate \eqref{eq:asymptotic_rate} is a decreasing function of $\sigma_0$. This motivates the introduction of the ``optimal'' degree $K$ polynomial $\sigma_K^{\Lambda}$ as the one that solves
\begin{equation}
    \sigma_K^{\Lambda} \triangleq \argmax_{\sigma \in \mathbb{R}_{K}[X]} \sigma(0) \quad \text{s.t.} \;\; \underset{\lambda \in \Lambda}{\sup} |\sigma(\lambda)| \leq 1\,. \label{eq:optimal_sigma}
\end{equation}
Using the above definition, we recover the $\sigma_1^{\Lambda}$ and $\sigma_2^{\Lambda}$ from Example \ref{ex:lambda_interval} and Proposition \ref{prop:lambda_union_two_interval}.

\paragraph{Finding the polynomial.} Finding an exact and explicit solution for the general $K$ and $\Lambda$ case is unfortunately out of reach, as it involves solving a system of $K$ non-linear equations. Here we describe an approximate approach. Let $\sigma_K^{\Lambda}(x) = \sum_{i=0}^K \sigma_ix^i$. We propose to discretize $\Lambda$ into $N$ different points $\{\lambda_j\}$, then solve the linear problem
\begin{equation}
    \max_{\sigma_i} \,\sigma_0 \quad \text{ s.t. } \;\; {\textstyle -1 \leq \sum_{i=0}^K \sigma_i \lambda_j^i \leq 1}, \quad\forall j=1,\,\ldots,\, N\,.
\end{equation}
To check the optimality, it suffices to verify that the polynomial $\sigma_K^{\Lambda}$ satisfies the \textit{equioscillation} property (Definition~\ref{def:equiosc}), as depicted in Figure \ref{fig:minimax_polynomials}.

\begin{Rem}[Relationship between optimal and minimax polynomials]\label{rem:relationship_btw_optimal_and_minimax}
    For later reference, we note that the optimal polynomial $\sigma_K^{\Lambda}$ is equivalent to finding a minimax polynomial on $\Lambda$ and to rescale it. More precisely,  $\sigma_K^{\Lambda}$ is optimal if and only if $\sigma_K^{\Lambda}/\sigma_K^{\Lambda}(0)$ is minimax.
\end{Rem}

\begin{figure*}
    \centering
    \includegraphics[width=\textwidth,trim=0.5cm 0cm 0cm 0cm, clip=true]{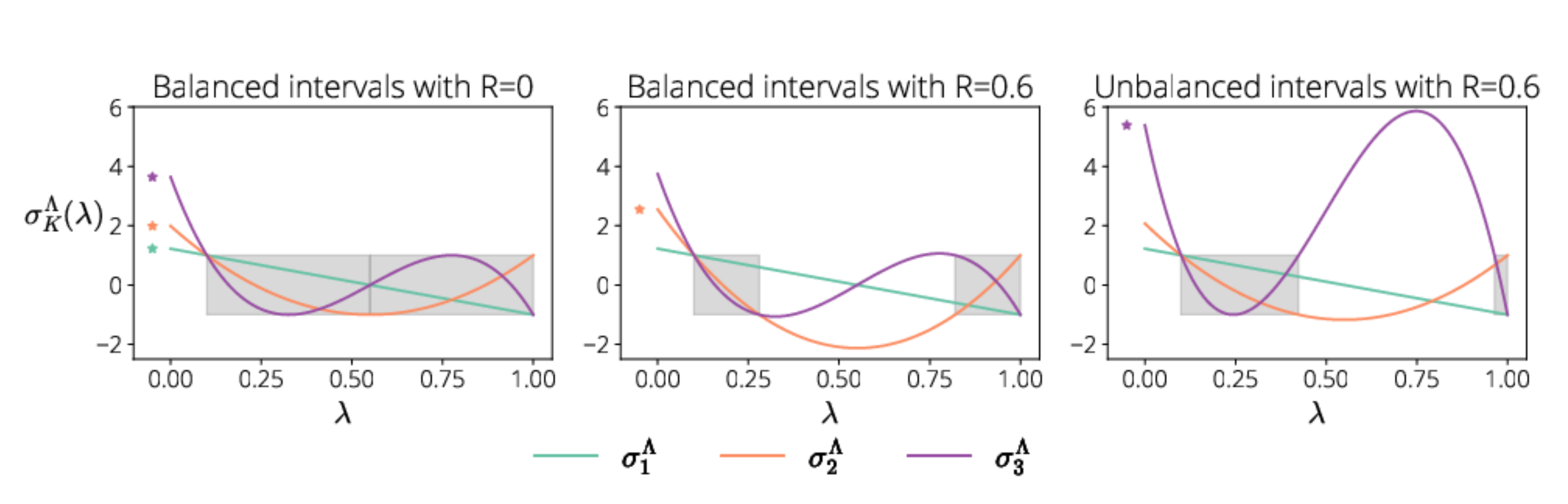}
    \vspace{-1em}\caption{Examples of optimal polynomials $\sigma_{K}^{\Lambda}$ from \eqref{eq:optimal_sigma}, all of them verifying the equioscillation property (Definition \ref{def:equiosc}).
    The ``$\star$'' symbol highlights the degree of $\sigma_K^{\Lambda}$ that achieves the best asymptotic rate $\tau^{\sigma_K^{\Lambda}}$ in \eqref{eq:asymptotic_rate} amongst all $K$ (see Section \ref{subsec:optim}). \textbf{(Left)} When $\Lambda$ is an unique interval, all 3 polynomials are equivalently optimal \tausigmaone=\tausigmatwo=\tausigmathree. \textbf{(Center)} When $\Lambda$ is the union of two intervals of the same size, the degree 2 polynomial is optimal \tausigmatwo>\tausigmathree>\tausigmaone. This is expected given the  result in Proposition \ref{prop:lambda_union_two_interval}. \textbf{(Right)} When $\Lambda$ is the union of two unbalanced intervals, the degree 3 polynomial instead achieves the best asymptotic rate \tausigmathree>\tausigmatwo>\tausigmaone (see Section~\ref{subsec:optim}).}
    \label{fig:minimax_polynomials}
\end{figure*}

\subsection{Cyclical Heavy Ball and (Non-)asymptotic Rates of Convergence} \label{ssec:resulting_algo}

We now describe the link between $\sigma_{K}^{\Lambda}$ and Algorithm \ref{algo:cyclical_chebyshev}. Using the recurrence for Chebyshev polynomials of the first kind in~\eqref{eq:parametrization_z}, we have $ \forall t = Kn,\, n\in\mathbb{N}^+_0$,

\begin{align*}
     \frac{T_{n+1}(\sigma_K^{\Lambda}(\lambda))}{T_{n+1}(\sigma_K^{\Lambda}(0))} &= 2\sigma_K^{\Lambda}(\lambda) \left[\frac{T_{n}(\sigma_K^{\Lambda}(\lambda))}{T_{n}(\sigma_K^{\Lambda}(0))} \right]\underbrace{\left[\frac{T_{n}(\sigma_K^{\Lambda}(0))}{T_{n+1}(\sigma_K^{\Lambda}(0))}\right]}_{=a_n}\\
     &\quad- \left[\frac{T_{n-1}(\sigma_K^{\Lambda}(\lambda))}{T_{n-1}(\sigma_K^{\Lambda}(0))} \right]\underbrace{\left[\frac{T_{n-1}(\sigma_K^{\Lambda}(0))}{T_{n+1}(\sigma_K^{\Lambda}(0))}\right]}_{=b_n}.
\end{align*}

It still remains to find an algorithm associated with this polynomial. To obtain one in the form of Algorithm \ref{algo:cyclical_heavy_ball}, one can use the stationary behavior of the recurrence. From \citep{scieur2020universal}, the coefficients $a_n$ and $b_n$ converge as $n\rightarrow \infty$ to their fixed-points $a_\infty$ and $b_\infty$. We therefore consider here an asymptotic polynomial $\bar{P}_{t}(\lambda;\sigma_K^{\Lambda})$, whose recurrence satisfies
\begin{equation} \label{eq:asymptotic_polynomial}
    \bar{P}_{t}(\lambda;\sigma_K^{\Lambda}) = 2a_\infty \sigma_K^{\Lambda}(\lambda) \bar{P}_{t-K}(\lambda;\sigma_K^{\Lambda})- b_\infty\bar{P}_{t-2K}(\lambda;\sigma_K^{\Lambda})\,.
\end{equation}
Similarly to $K=1$, where this limit recursion corresponds to PHB, this recursion corresponds to an instance of Algorithm~\ref{algo:cyclical_chebyshev} (see~\Cref{prop:sequence_stepsizes} below), further motivating the cyclical heavy ball algorithm.

The following theorem is the main result of this section and characterizes the convergence rate of Algorithm \ref{algo:cyclical_heavy_ball} for arbitrary momentum and step-size sequence $\{h_i\}_{i\in\llbracket 1,K\rrbracket}$.

\begin{restatable}{Th}{generalrateconvergence}
    \label{thm:general_rate_convergence}
   With an arbitrary momentum $m$ and an arbitrary sequence of step-sizes $\{h_i\}$ , the worst-case rate of convergence $ 1 - \tau $ of Algorithm \ref{algo:cyclical_heavy_ball} on $\mathcal{C}_\Lambda$ is
    \begin{equation}
           \left\{
            \begin{array}{ll}
               \!\!\!\! \sqrt{m}
                & 
                \mathrm{if~} \sigma_{*}  \leq 1
                \\
                \!\!\!\! \sqrt{m}\left( \sigma_{*} + \sqrt{\sigma_{*}^2 - 1} \right)^{K^{-1}}
                & 
                \mathrm{if~} \sigma_{*}  \in\left(1,\; \dfrac{1+m^K}{2\left(\sqrt{m}\right)^K}\right)
                \\
                \!\!\!\! \geq 1\,(\text{no convergence)}
                &  
                \mathrm{if~} \sigma_{*}  \geq \dfrac{1+m^K}{2\left(\sqrt{m}\right)^K}\, , 
            \end{array}
            \right. \,
        \end{equation}
    where $\sigma_{*}\triangleq \underset{\lambda \in \Lambda}{\sup} |\sigma(\lambda;\{h_i\}, m)|$, $\sigma(\lambda;\{h_i\},m)$ is the $K$-degree polynomial
    \begin{equation}
    \label{eq:sigma_stepsize}
        \sigma(\lambda;\{h_i\}, m) \triangleq \frac{1}{2} \mathrm{Tr} \left(M_1M_2\ldots M_K\right),
    \end{equation}
    and $M_i = \begin{bmatrix}
            \frac{1+m- h_{K-i}\lambda}{\sqrt{m}}   & -1 \\ 
            1                           &  0 
        \end{bmatrix} $.
\end{restatable}

By optimizing over these parameters, we obtain the Algorithm \ref{algo:cyclical_chebyshev}, a method associated to \eqref{eq:asymptotic_polynomial}, whose rate is described in Proposition~\ref{prop:sequence_stepsizes}. All proofs can be found in Appendix \ref{apx:derivation-of-heavy-ball-withstep-sizes-cycle}.

\begin{algorithm}[ht] \label{algo:cyclical_chebyshev}
    \caption{Cyclical (arbitrary $K$) heavy ball with optimal parameters}
    \SetAlgoLined
    \textbf{Input:} Eigenvalue localization $\Lambda$, cycle length $K$, initialization $x_0$. \\
    \textbf{Preprocessing:}
    \begin{enumerate}
        \item Find the polynomial $\sigma_K^{\Lambda}$ such that it satisfies \eqref{eq:optimal_sigma}.
        \item Set step-sizes $\{h_i\}_{i=0, \ldots, K-1}$ and momentum $m$  that satisfy resp. equations \eqref{eq:equation_sequence_stepsizes} and \eqref{eq:optimal_momentum}.
    \end{enumerate}
    \textbf{Set} $ x_1 ~ = x_0 - \mfrac{h_0}{1 + m}\nabla f(x_0) $\\
    \For{$t = 1,\, 2,\,\ldots$}{
        \vspace{-2ex}
        \begin{align*}
            x_{t+1} & = ~ x_t - h_{\text{mod}(t,K)} \nabla f(x_t) + m(x_{t}-x_{t-1})
        \end{align*}
        \vspace{-2ex}
    }
\end{algorithm}

\begin{restatable}{Prop}{sequencestepsizes}
    \label{prop:sequence_stepsizes}
     Let $\sigma(\lambda;\{ h_i\},m)$ be the polynomial defined by \eqref{eq:sigma_stepsize}, and $\sigma_K^{\Lambda}$ be the optimal link function of degree $K$ defined by \eqref{eq:optimal_sigma}. If the momentum $m$ and the sequence of step-sizes $\{h_i\}$ satisfy
    \begin{equation}
        \label{eq:equation_sequence_stepsizes}
        \sigma(\lambda;\{ h_i\}, m) = \sigma_K^{\Lambda}(\lambda)\,,
    \end{equation}
    then \textbf{1)} the parameters are optimal, in the sense that they minimize  the asymptotic rate factor from Theorem \ref{thm:general_rate_convergence}, \textbf{2)} the optimal momentum parameter is
    \begin{equation} \label{eq:optimal_momentum}
        \textstyle m = \big(\sigma_0 - \sqrt{\sigma_0^2-1}\big)^{2/K}, \quad \text{where}\;\sigma_0 = \sigma_K^{\Lambda}(0)\,,
    \end{equation}
    \textbf{3)} the iterates from Algo. \ref{algo:cyclical_chebyshev} with parameters $\{h_i\}$ and $m$ form a polynomial with recurrence \eqref{eq:asymptotic_polynomial}, and \textbf{4)}~Algorithm~\ref{algo:cyclical_chebyshev} achieves the worst-case rate $r_t^\text{Alg. 3}$ and the asymptotic rate factor $1-\tau^\text{Alg. 3}$ 
    \begin{gather}
        \textstyle r_t^\text{Alg. 3} = O\left(t\left(\sigma_0 - \sqrt{\sigma_0^2-1}\right)^{t/K}\right),\\
        1 - \tau^\text{Alg. 3} = \left(\sigma_0 - \sqrt{\sigma_0^2-1}\right)^{1/K}\,. \nonumber
    \end{gather}
\end{restatable}

\paragraph{Solving the system \eqref{eq:equation_sequence_stepsizes}}
The system is constructed by identification of the coefficients in both polynomials $\sigma^{\Lambda}_K$ and $\sigma(\lambda;\{h_i\},m)$, which can be solved using a naive grid-search for instance. We are not aware of any efficient algorithm to solve this system exactly, although it is possible to use iterative methods such as steepest descent or Newton's method.

\subsection{Best Achievables Worst-case Guarantees on \texorpdfstring{$\mathcal{C}_{\Lambda}$}{C Lambda}}\label{subsec:optim}

This section discusses the (asymptotic) optimality of Algorithm \ref{algo:cyclical_chebyshev}. In Section \ref{ssec:finding_sigma}, the polynomial $P_{t}(\,\cdot\,; \sigma_K^\Lambda)$ was written as a composition of Chebyshev polynomials with $\sigma_K^{\Lambda}$, defined in \eqref{eq:optimal_sigma}. The best $K$ is chosen as follows: we solve \eqref{eq:optimal_sigma} for several values of $K$, then pick the smallest $K$ among the minimizers of~\eqref{eq:asymptotic_rate}. However, following such steps does not guarantee that the polynomial $P_{t,K}^\Lambda$ is \textit{minimax}, as it is not guaranteed to minimize the worst-case rate $\sup_{\lambda\in\Lambda}|P_t(\lambda)|$ (see~\eqref{eq:rate_convergence_polynomial}).

We give here an optimality certificate, linked to a generalized version of \textit{equioscillation}. In short, if we can find $K$ non overlapping intervals (more formally, whose interiors are disjoint)  $\Lambda_i$ in $\Lambda$ such that $\sigma_{K}^{\Lambda}(\Lambda_i)=[-1, 1]$ then $P_{t,K}^{\Lambda}$ is minimax for $t = nK$, $n\in\mathbb{N}_0^+$. The precise result is in Theorem~\ref{thm:optimality_from_equioscillation}. A direct consequence is the asymptotic optimality of Algorithm \ref{algo:cyclical_chebyshev}.

We note that $\sigma_K^{\Lambda}$ might not exist for a given $\Lambda$. A complete characterization of the set $\Lambda$ for which $\sigma_K^{\Lambda}$ exists is out of the scope of this paper. A partial answer is given in \citep{fischer2011polynomial} when $\Lambda$ is the union of two intervals but the general case remains open.

\section{Local Convergence for Non-Quadratic Functions}\label{sec:local-convergence}

When $f$ is twice-differentiable, we show local convergence rates of Algorithm \ref{algo:cyclical_heavy_ball} (see proof in \Cref{apx:beyond-quadratic-objective:-local-convergence-of-cycling-methods}). As with Polyak heavy ball acceleration, these results are local, as the only known convergence results for Polyak heavy ball beyond quadratic objectives do not lead to an acceleration with respect to Gradient descent without momentum \citep[See][Theorem 4]{ghadimi2015global}.
Moreover, it is possible to find pathological counter-examples and a specific initialization for which the method does not converge globally. \citet[Figure 7]{lessard2016analysis} provides such a counter-example.
Note the latest is \emph{not} twice differentiable, but that a twice differentiable counter-example can be derived from the latest, using for instance convolutions.

\begin{restatable}[Local convergence]{Th}{localconvergencenonquadratic} \label{thm:local_convergence_non_quadratic}
    Let $f: \mathbb{R}^d \mapsto \mathbb{R}$ be a twice continuously differentiable function, $x_*$ a local minimizer, and $H$ be the Hessian of $f$ at $x_*$ with $Sp(H) \subseteq \Lambda$. Let $x_t$ denote the result of running Algorithm \ref{algo:cyclical_heavy_ball} with parameters $h_1, h_2, \cdots, h_K, m$, and let $1 - \tau$ be the linear convergence rate on the quadratic objective \eqref{eq:quad_problem}. Then we have
    \begin{gather}
        \forall \varepsilon>0, \exists ~ \mathrm{open~set} ~ V_{\varepsilon} : x_0,\,x_*\in V_{\varepsilon} \nonumber \\
        \implies \|x_t - x_*\| = O((1 - \tau+\varepsilon)^t)\|x_0-x_*\|.
    \end{gather}
\end{restatable}

where $\|\cdot\|$ denotes the Euclidean norm.

In short, when Algorithm \ref{algo:cyclical_heavy_ball} is guaranteed to converge at rate $1-\tau$ on~\eqref{eq:quad_problem}, then the convergence rate on a nonlinear functions can be arbitrary close to $1-\tau$ when $x_0$ is sufficiently close to $x_*$.

\section{Experiments}\label{sec:experiments}

\begin{figure*}
    \centering
    \includegraphics[width=\linewidth]{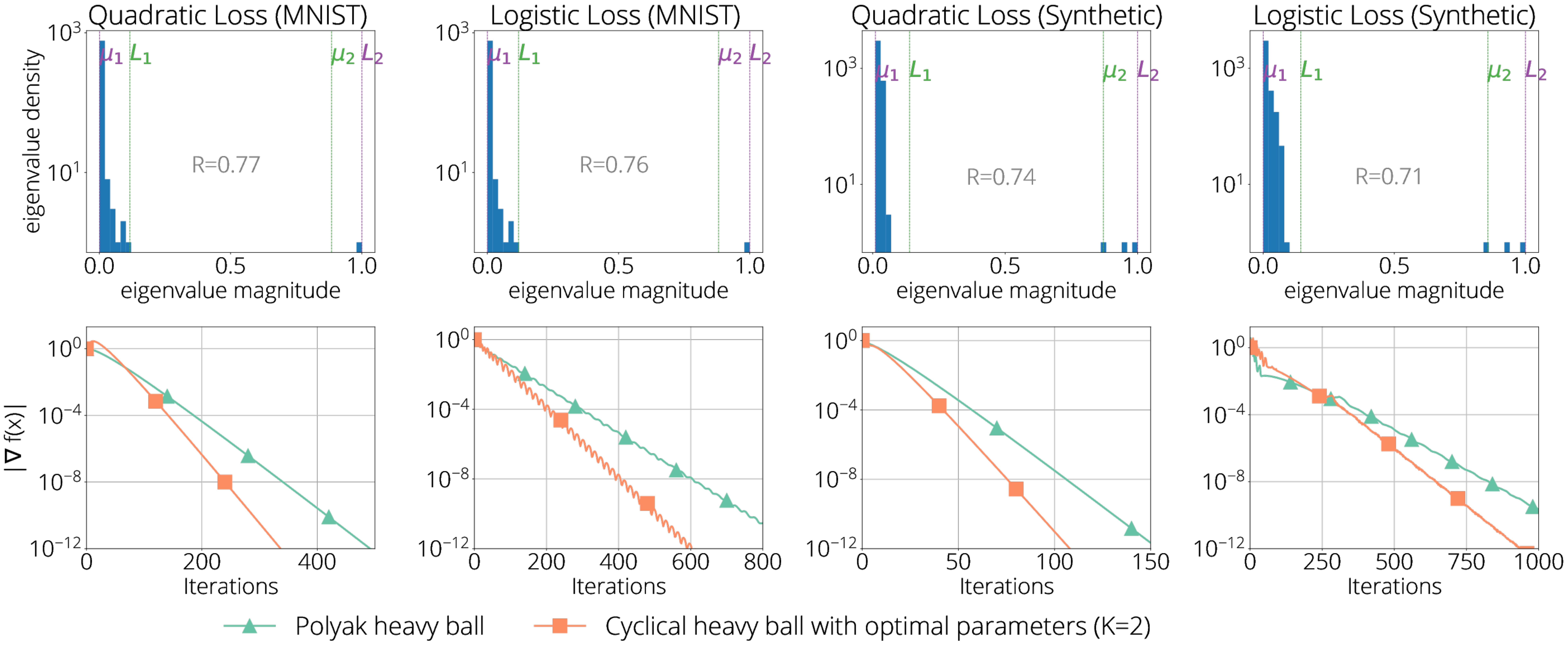}
    \caption{\emph{Hessian Eigenvalue histogram (top row) and Benchmarks (bottom row)}. The {\bfseries top row} shows the Hessian eigenvalue histogram at optimum for the 4 considered problems, together with the interval boundaries $\muone < \Lone < \mutwo < \Ltwo$ for the two-interval split of the eigenvalue support described in Section \ref{sec:alternating_hb}. In all cases, there's a non-zero gap radius $R$. This is shown in the {\bfseries bottom row}, where we compare the suboptimality in terms of gradient norm as a function of the number of iterations. As predicted by the theory, the non-zero gap radius translates into a faster convergence of the cyclical approach, compared to PHB in all cases. The improvement is observed on both quadratic and logistic regression problems, even through the theory for the latter is limited to \emph{local} convergence. }
    \label{fig:benchmarks}
\end{figure*}
In this section we present an  empirical comparison of the cyclical heavy ball method for different length cycles across 4 different problems. We consider two different problems, quadratic and logistic regression, each applied on two datasets, the MNIST handwritten digits \citep{lecun2010mnist} and a synthetic dataset. The results of these experiments, together with a histogram of the Hessian's eigenvalues are presented in Figure \ref{fig:benchmarks} (see caption for a discussion).

\textbf{Dataset description.} The MNIST dataset consists of a data matrix $A$ with $60000$ images of handwritten digits each one with $28 \times 28 = 784$ pixels. The \emph{synthetic} dataset is generated according to a spiked covariance model \citep{johnstone2001distribution}, which has been shown to be an accurate model of covariance matrices arising for instance in spectral clustering \citep{couillet2016kernel} and deep networks \citep{pennington2017nonlinear, granziol2020explaining}.
In this model, the data matrix $A = X Z$ is generated from a $m \times n$ random Gaussian matrix $X$ and an $m \times m$ deterministic matrix $Z$. In our case, we take $n=1000, m=1200$ and $Z$ is the identity where the first three entries are multiplied by 100 (this will lead to three outlier eigenvalues). We also generate an $n$-dimensional target vector $b$ as $b = Ax$ or $b = \text{sign}(A x)$ for the quadratic and logistic problem respectively.

\textbf{Objective function} For each dataset, we consider a quadratic and a logistic regression problem, leading to 4 different problems. All problems are of the form $\min_{x \in \mathbb{R}^p} \tfrac{1}{n}\sum_{i=1}^n \ell(A_i^\top x, b_i) + \lambda \|x\|^2 $, where $\ell$ is a quadratic or logistic loss, $A$ is the data matrix and $b$ are the target values. We set the regularization parameter to $\lambda=10^{-3} \|A\|^2$. For logistic regression, since guarantees only hold at a neighborhood of the solution (even for the  1-cycle algorithm), we initialize the first iterate as the result of 100 iteration of gradient descent. In the case of logistic regression, the Hessian eigenvalues are computed at the optimum.

\section{Conclusion}\label{sec:conclusions}
This work is motivated by two recent observations from the optimization practice of machine learning. First, cyclical step-sizes have been shown to enjoy excellent empirical convergence~\citep{loshchilov2016sgdr,smith2017cyclical}. Second, \emph{spectral gaps} are pervasive in the Hessian spectrum of deep learning models~\citep{sagun2017empirical, papyan2018full, ghorbani2019investigation, papyan2019measurements}. Based on the simpler context of quadratic convex minimization, we develop a convergence-rate analysis and optimal parameters for the heavy ball method with cyclical step-sizes. This analysis highlights the regimes under which cyclical step-sizes have faster rates than classical accelerated methods. Finally, we illustrate these findings through numerical benchmarks.

\paragraph{Main Limitations.}
In \Cref{sec:alternating_hb} we gave explicit formulas for the optimal parameters in the case of the 2-cycle heavy ball algorithm. These formulas depend not only on extremal eigenvalues---as is usual for accelerated methods---but also on the spectral gap $R$. The gap can sometimes be estimated after computing the top eigenvalues (e.g. top-2 eigenvalue for MNIST). However, in general, there is no guarantee on how many eigenvalues are needed to estimate it and it must sometimes be seen as hyperparameter. Note \Cref{thm:rate_factor_alternating_hb} provides a convergence analysis also for non-optimal parameters, which would give accelerated convergence rates when doing a coarse grid-search over parameters as it is often done in empirical works.

Another limitation is the fact global convergence results rely heavily on the quadratic assumption which is quite different from our motivation, namely optimizing neural networks. Even if we provide local convergence guarantee in \Cref{sec:local-convergence}, we are not able to estimate the size of the optimum neighborhood for which \Cref{thm:local_convergence_non_quadratic} holds.

Another limitation regards long cycles.
For cycles longer than 2, we gave an implicit formula to set the optimal parameters (Proposition \ref{prop:sequence_stepsizes}). This involves solving a set of non-linear equations whose complexity increases with the cycle length. That being said, cyclical step-sizes might significantly enhance convergence speeds both in terms of worst-case rates and empirically, and this work advocates that new tuning practices involving different cycle lengths might be relevant.

\subsubsection*{Acknowledgements}
    We thank Courtney Paquette for very useful proofreading and feedback. The work of B. Goujaud and A. Dieuleveut is partially supported by ANR-19-CHIA-0002-01/chaire SCAI, and Hi!Paris. A. Taylor acknowledges support from the European Research Council (grant SEQUOIA 724063). This work was partly funded by the French government under management of Agence Nationale de la Recherche as part of the “Investissements d’avenir” program, reference ANR-19-P3IA-0001 (PRAIRIE 3IA Institute).

\bibliographystyle{plainnat}
\bibliography{references}
\clearpage
\appendix
\onecolumn
\aistatstitle{Super-Acceleration with Cyclical Step-sizes: \\
Supplementary Materials}

\section*{Organization of the appendix}

    The appendix contains all proofs that were not presented in the main core of the paper. We also detail all examples, and provide some complementary elements.
    
    \Cref{apx:link-between-first-order-algorithms-and-polynomials} details the existing link between first order methods and family of ``residual polynomials''.
    This term refers in all the appendix to the polynomials which value in 0 is 1.
    
    In \Cref{apx:optimal-methods-for-strongly-convex-smooth-quadratic-objective}, we recall some well known optimal methods for $L$-smooth $\mu$-strongly convex quadratic minimization (i.e., when the spectrum is contained in a single interval $\Lambda=[\mu,L]$). Its purpose is exclusively to recall well-known foundation of optimization that are those algorithms and their construction.
    
    In \Cref{apx:minimax-polynomials-and-equioscillation-property}, we recall the polynomial formulation of the optimal method design problem, as well as a fundamental property, called ``equioscillation'', to characterize the solution of this problem.
    
    In Appendix~\ref{apx:cycling}, we provide all proofs related to cyclic step-sizes. In particular,
    \begin{itemize}
        \item In \Cref{apx:derivation-of-optimal-algorithm-withstep-sizes}, we derive the optimal algorithm in a case where $\Lambda$ is the union of 2 intervals of the same size (See~\eqref{eq:def_lambda}). This leads to the use of alternating step-sizes. The resulting algorithm has a stationary form which is Algorithm~\ref{algo:cyclical_heavy_ball}.
        \item Therefore, in~\Cref{apx:derivation-of-heavy-ball-withstep-sizes-cycle}, we study the heavy ball with cycling step-sizes (Algorithm~\ref{algo:cyclical_heavy_ball}).
        \item In \Cref{apx:case-of-a-2-step-sizes-cycle} and \Cref{apx:case-of-a-3-step-sizes-cycle}, we use our results to design methods with cycles of lengths $K=2$ and $K=3$. For those cases, we provide a more elegant formulation of the results.
    \end{itemize}
    
    In \Cref{apx:beyond-quadratic-objective:-local-convergence-of-cycling-methods}, we provide a proof of Theorem~\ref{thm:local_convergence_non_quadratic} (local behavior beyond quadratics) and in~\Cref{apx:experiments}, we provide some information about the code we used for the experiments in quadratic and non quadratic settings.
    
    Finally, in~\Cref{apx:comparison_with_Oymak} we discuss similarities and differences with \citet{oymak2021super}.
    
	\hypersetup{linkcolor = black}
	\setlength\cftparskip{2pt}
	\setlength\cftbeforesecskip{2pt}
	\setlength\cftaftertoctitleskip{3pt}
	\addtocontents{toc}{\protect\setcounter{tocdepth}{2}}
	\setcounter{tocdepth}{1}
	\tableofcontents
	\hypersetup{linkcolor=blue}

\section{Relationship between first order methods and polynomials}\label{apx:link-between-first-order-algorithms-and-polynomials}
    In this section we prove some results on the relationship between polynomials and first order methods for quadratic minimization, which is the starting point for our theoretical framework. This relationship is classical and was exploited  by \citet{rutishauser1959theory, nemirovsky1992information,nemirovski1995information}), to name a few.
    The following proposition makes this relationship precise:
    \linkalgopoly*
    
    \begin{proof}
    We successively prove both directions of the equivalence.
    
    ($\Longrightarrow$) \textit{Given a first order method, we can find a sequence of polynomials $(P_t)_{t\in\mathbb{N}}$ such that, for a given quadratic function $f$ of Hessian $H$ and a given starting point $x_0$, the iterates $x_t$ verify 
    \[
        x_t - x_* = P_t(H)(x_0 - x_*).
    \] 
    Moreover, The polynomials sequence $(P_t)_{t\in\mathbb{N}}$ verifies the relations 
    \[
        \mathrm{deg}(P_{t+1}) \leq \underset{k\leq t}{\max}\mathrm{~deg}(P_k) + 1 \quad \text{and} \quad P_t(0)=1.
    \]
    }
    
    We proceed by induction:
    
    \paragraph{Initial case.} Let $t=0$. Then for any first order method we have the trivial relationship 
    \[ 
        x_0 - x_* = P_0(H)(x_0 - x_*)  \quad \text{with} \quad P_0 = 1.
    \] 
    This proves the implication for $t=0$, as $P_0$ is a degree 0 polynomial satisfying $P_0(0)=1$.

    \paragraph{Recursion.} Let $t\in\mathbb{N}$. We assume the following statement true, 
    \[
        \forall\; k \leq t, \quad x_k - x_* = P_k(H)(x_0 - x_*) \quad \text{with} \quad P_k(0)=1.
    \]
    We now prove this statement is also true for $t+1$. Since $x_{t+1} \in x_0 + \text{span}\{ \nabla f(x_0),\ldots, \nabla f(x_t) \}$, there exists a family $(\gamma_{t+1, k})_{k\in\llbracket 0, t\rrbracket}$ such that
    \begin{equation}
        x_{t+1} = x_0 - \gamma_{t+1, 0} \nabla f(x_0) - \cdots - \gamma_{t+1, t} \nabla f(x_t).
    \end{equation}
    Then, by the induction hypothesis we have:
    \begin{align*}
        x_{t+1} - x_* &~ = x_0 - x_* - \gamma_{t+1, 0} H (x_0 - x_*) - \cdots - \gamma_{t+1, t} H (x_t - x_*) \\
        &~ = x_0 - x_* - \gamma_{t+1, 0} H P_0(H) (x_0 - x_*) - \cdots - \gamma_{t+1, t} H P_t(H) (x_0 - x_*) \\
        &~ \triangleq P_{t+1}(H) (x_0 - x_*)\,.
    \end{align*}
    We observe that the latest polynomial has a degree at most 1 plus the highest degree of $(P_k)_{k \leq t}$ and that $P_{t+1}(0) = 1$ (since $P_{t+1}$ is defined as 1 plus some polynomial multiple of the polynomial $X$), which concludes the proof.

    ($\Longleftarrow$): \textit{From a family of polynomials $(P_t)_{t\in\mathbb{N}}$, with
    \begin{equation} \label{eq:sequence_poly_degree_app}
        \mathrm{deg}(P_{t+1}) \leq \underset{k\leq t}{\max}\mathrm{~deg}(P_k) + 1 \quad \text{and} \quad P_t(0)=1,
    \end{equation}
    we can obtain a first order method such that, for any quadratic $f$ (and its Hessian $H$) and any starting point $x_0$, we verify
    \[
        \forall t\in\mathbb{N}, x_t - x_* = P_t(H)(x_0-x_*).
    \]
    }
    
    Let the sequence $(P_t)_{t\in\mathbb{N}}$ verifies \eqref{eq:sequence_poly_degree_app} for all $t\in\mathbb{N}$. Let
    \[ 
        d=\underset{t'\leq t}{\max}\,\mathrm{deg}(P_{t'}).
    \]
    
    A gap in the sequence of degrees would stand in contradiction with our assumptions.
    
    Since, there is no gap in degree, for any $d' \leq d$ there exists $t' \leq t$ such that $\mathrm{deg}(P_{t'}) = d'$, and therefore  $\mathrm{Span}((P_{k})_{k \leq t}) = \mathbb{R}_d[X]$.
    
    Moreover, we know $P_{t+1}$ has a degree at most $d+1$ and $P_{t+1}(0)=1$, so $\frac{1 - P_{t+1}(X)}{X} \in \mathbb{R}_d[X]$.
    
    This proves the existence of $(\gamma_{t+1, k})_{k\in\llbracket 0, t\rrbracket}$ such that 
    \begin{equation}
        \frac{1 - P_{t+1}(X)}{X} = \gamma_{t+1, 0} P_0(X) + \cdots + \gamma_{t+1, t} P_t(X).
    \end{equation}
    Then, defining
    \begin{equation}\label{eq:recursion_def}
        x_{t+1} = x_0 - \gamma_{t+1, 0} \nabla f(x_0) - \cdots - \gamma_{t+1, t} \nabla f(x_t)\,,
    \end{equation}
    we have
    \begin{align}
        x_{t+1} - x_* & ~ = x_0 - x_* - H \left(\gamma_{t+1, 0} (x_0-x_*) + \cdots + \gamma_{t+1, t} (x_t-x_*)\right) \\
        & ~ = \left(1 - X\left( \gamma_{t+1, 0} P_0(X) + \cdots + \gamma_{t+1, t} P_t(X)\right) \right)(H)(x_0 - x_*) \\
        & ~ = P_{t+1}(H)(x_0 - x_*)\,.
    \end{align}
    Defining $x_t$ for all $t$ according to \eqref{eq:recursion_def} gives an algorithm that has as associated residual polynomials $(P_t)_{t\in\mathbb{N}}$.
\end{proof}
    
    The above proposition can be used to obtain worst-case rates for first order methods by bounding their associated polynomials. Indeed, using the Cauchy-Schwartz inequality in \eqref{eq:link_polynomial} leads to
    \begin{equation}
        \|x_t-x_*\| \leq \sup_{\lambda \in \Lambda} |P_t(\lambda)|\; \|x_0-x_*\| \quad \Longrightarrow \quad r_t = \sup_{\lambda \in \Lambda} |P_t(\lambda)|, \quad \text{where } P(0)=1\,.
    \end{equation}
    Therefore, finding the algorithm with the fastest worst-case rate can be equivalently framed as the problem of finding the residual polynomial with smallest value on the eigenvalue support $\Lambda$.

    Then, finding the fastest algorithm is equivalent of finding, for each $t \geq 0$, the polynomial of degree $t$ that reaches the smallest infinite norm on the set $\Lambda$.
    Therefore we introduce the notion of \textit{minimax polynomial} (Definition \ref{def:optimally_small_polynomial}) over a set $\Lambda$ as the one that reaches the smallest maximal value over $\Lambda$ among a set of polynomial of fixed degree and $P(0)=1$.
    \begin{Def}[Minimax polynomial of degree $t$ over $\Lambda$] \label{def:optimally_small_polynomial}
        For any, $t\geq 0$, and any relatively compact (i.e. bounded) set $\Lambda\subset\mathbb{R}$, the \textit{minimax polynomial of degree $t$ over $\Lambda$}, written  $Z_{t}^{\Lambda}$, is defined as
        \begin{equation}
             Z_{t}^{\Lambda} \triangleq \underset{P \in \mathbb{R}_t[X]}{\mathrm{argmin~}} \;\; \underset{\lambda \in \Lambda}{\sup} |P(\lambda)|, \quad \mathrm{subject~to~} \;\;P(0) = 1\,.
            \label{eq:optimally_small_polynomial}
        \end{equation}
    \end{Def}
    
\section{Optimal methods for strongly convex and smooth quadratic objective 
}\label{apx:optimal-methods-for-strongly-convex-smooth-quadratic-objective}

In this section, for sake of completness, we revisit some classical methods, described in e.g. \citep{polyak1964some, distillblog, fabian2020blog, fabian2021blog}, that are optimal when the Hessian eigenvalues are contained in a single interval of the form $\Lambda=[\mu,L]$. To make this setup explicit, we will denote the optimal polynomials $\sigma_1^{\Lambda}$ and $Z_t^{\Lambda}$ (respectively defined in~\Cref{eq:optimal_sigma} and \Cref{eq:optimally_small_polynomial}) by $\sigma_1^{[\mu,L]}$, and $Z_t^{[\mu,L]}$.
    
    As mentioned in Example~\ref{ex:lambda_interval},  the minimax polynomial $Z_t^{[\mu, L]}$ is 
    \[
        Z_t^{[\mu, L]}(\lambda) = \frac{T_t(\sigma_1^{[\mu,L]}(\lambda))}{T_t(\sigma_1^{[\mu,L]}(0))},
    \] 
    where $T_t$ denotes the $t^{th}$ Chebyshev polynomial (See e.g. \citet{chebyshev1853theorie}) and $\sigma_1^{[\mu,L]}$ the affine function $\sigma(\lambda) \triangleq \frac{L+\mu}{L-\mu} - \frac{2}{L-\mu}\lambda$ that maps $[\mu, L]$ onto $[-1, 1]$. This can be seen a consequence of the more general \textit{equioscillation} discussed in Appendix \ref{apx:minimax-polynomials-and-equioscillation-property}. The next section presents one method which has $Z_t^{[\mu, L]}$ as associated residual polynomial. This method is known as the Chebyshev semi-iterative method.

    \subsection{Chebyshev semi-iterative method}\label{subsec:chebyshev-method}

        The algorithm follows the three terms pattern from \Cref{eq:optimal_poly_smooth_strong_convex} to iteratively form $Z_1^{\Lambda}, \ldots,Z_t^{\Lambda}$.
    
        \begin{algorithm}[H]
            \caption{Chebyshev semi-iterative method \citep{golub1961chebyshev}} \label{algo:cheby_method}
            \SetAlgoLined
            \textbf{Input:} $x_0$ \\
            \textbf{Initialize:} $\omega_0 = 2$
            ~\\
            $x_{1} = x_0 - \frac{2}{L+\mu} \nabla f(x_0)$\;
            ~\\
            \For{$t = 1, \dots$}{
                $\omega_{t+1} = \left(1-\frac{1}{4}\big(\frac{1-\kappa}{1+\kappa}\big)^2\omega_t\right)^{-1}$ \;
                $x_{t+1} = x_t - \frac{2}{L+\mu} \omega_t  \nabla f(x_t) + (\omega_t-1) (x_t - x_{t-1}) $ \;
            }
        \end{algorithm}
        
        \begin{Th}
            The iterates produced by the Chebyshev semi-iterative method verify
            \begin{equation} \label{eq:cheby_polynomial_app}
                x_t - x_* = \frac{T_t(\sigma_1^{[\mu,L]}(H))}{T_t(\sigma_1^{[\mu,L]}(0))} (x_0 - x_*) \quad\text{for all} \;\;t\in\mathbb{N}.
            \end{equation}
                
            Furthermore, this method enjoys a worst-case rate of the form
            \begin{equation}
            \|x_t - x_*\| \leq \frac{1}{T_t(\sigma_1^{[\mu,L]}(0))} \|x_0 - x_*\| = O\left( \left( \frac{1 - \sqrt{\kappa}}{1 + \sqrt{\kappa}} \right)^t \right)\,.
            \end{equation}
        \end{Th}

        \begin{proof}
            Consider first an algorithm whose iterates verify \eqref{eq:cheby_polynomial_app}. Then using the Cauchy-Schwartz inequality and known bounds of Chebyshev polynomials, we can show the following rate
            \begin{align*}
                \|x_t - x_*\| & \leq \frac{\underset{\lambda \in [\mu, L]}{\sup} |T_t(\sigma_1^{[\mu,L]}(\lambda))|}{T_t(\sigma_1^{[\mu,L]}(0))} \|x_0 - x_*\| \\
                & = \frac{1}{T_t\left(\frac{1+\kappa}{1-\kappa}\right)} \|x_0 - x_*\| & \mathrm{since} \, \underset{x \in [-1, 1]}{\sup} |T_t(x)| = 1\\
                & \leq 2 \left(\frac{1-\sqrt{\kappa}}{1+\sqrt{\kappa}}\right)^t \|x_0 - x_*\| & \mathrm{since} \, T_t(x) \geq \frac{\left(x + \sqrt{x^2-1}\right)^t}{2}, \forall x\notin (-1, 1) \,.
            \end{align*}
            
            It remains to prove that \Cref{algo:cheby_method} is the one that achieves the property \eqref{eq:cheby_polynomial_app}.
            Using the recursion verified by Chebyshev polynomials
            \begin{equation}\label{eq:cheby_polynomials}
                T_{t+1}(x) = 2xT_{t}(x) - T_{t-1}(x),
            \end{equation}
            we have
            \begin{align*}
                x_{t+1} - x_* & = \frac{T_{t+1}(\sigma_1^{[\mu,L]}(H))}{T_{t+1}(\sigma_1^{[\mu,L]}(0))} (x_0 - x_*) \\
                & = \frac{2\sigma_1^{[\mu,L]}(H)T_{t}(\sigma_1^{[\mu,L]}(H))(x_0 - x_*) - T_{t-1}(\sigma_1^{[\mu,L]}(H))(x_0 - x_*)}{T_{t+1}(\sigma_1^{[\mu,L]}(0))} \\
                & = \frac{2\sigma_1^{[\mu,L]}(H)T_{t}(\sigma_1^{[\mu,L]}(0))}{T_{t+1}(\sigma_1^{[\mu,L]}(0))}(x_t - x_*) - \frac{T_{t-1}(\sigma_1^{[\mu,L]}(0))}{T_{t+1}(\sigma_1^{[\mu,L]}(0))}(x_{t-1} - x_*) \\
                & = \frac{2\sigma_1^{[\mu,L]}(0)T_{t}(\sigma_1^{[\mu,L]}(0))}{T_{t+1}(\sigma_1^{[\mu,L]}(0))}\left(I - \frac{2}{L+\mu}H\right)(x_t - x_*) - \frac{T_{t-1}(\sigma_1^{[\mu,L]}(0))}{T_{t+1}(\sigma_1^{[\mu,L]}(0))}(x_{t-1} - x_*)\,.
            \end{align*}
            Let's introduce $\omega_t \triangleq \frac{2\sigma_1^{[\mu,L]}(0)T_{t}(\sigma_1^{[\mu,L]}(0))}{T_{t+1}(\sigma_1^{[\mu,L]}(0))}$.
            Then $\omega_0 = 2$ and by Chebyshev recursion (\Cref{eq:cheby_polynomials}), $\omega_t - 1 = \frac{T_{t-1}(\sigma_1^{[\mu,L]}(0))}{T_{t+1}(\sigma_1^{[\mu,L]}(0))}$.
            With this notation we can write the above identity more compactly as
            \begin{align*}
                x_{t+1} - x_* & = \omega_t \left(I - \frac{2}{L+\mu}H\right)(x_t - x_*) - (\omega_t - 1) (x_{t-1} - x_*) \\
                & = x_t - \frac{2}{L+\mu} \omega_t  \nabla f(x_t) + (\omega_t-1) (x_t - x_{t-1})\,.
            \end{align*}
            
            It remains to find a recursion on $\omega_t$ to make its use tractable.
            Using one more time the Chebyshev recursion \Cref{eq:cheby_polynomials},
            \begin{align*}
                \omega_{t}^{-1} & = \frac{T_{t+1}(\sigma_1^{[\mu,L]}(0))}{2\sigma_1^{[\mu,L]}(0)T_{t}(\sigma_1^{[\mu,L]}(0))} \\
                & = \frac{2\sigma_1^{[\mu,L]}(0)T_{t}(\sigma_1^{[\mu,L]}(0)) - T_{t-1}(\sigma_1^{[\mu,L]}(0))}{2\sigma_1^{[\mu,L]}(0)T_{t}(\sigma_1^{[\mu,L]}(0))} \\
                & = 1 - \frac{1}{4\sigma_1^{[\mu,L]}(0)^2} \frac{2\sigma_1^{[\mu,L]}(0) T_{t-1}(\sigma_1^{[\mu,L]}(0))}{T_{t}(\sigma_1^{[\mu,L]}(0))} \\
                & = 1 - \frac{1}{4\sigma_1^{[\mu,L]}(0)^2}\omega_{t-1},
            \end{align*}
            which can finally be written as
            \begin{equation*}
                \omega_{t+1} = \frac{1}{1-\frac{1}{4}\left(\frac{1-\kappa}{1+\kappa}\right)^2\omega_t},
            \end{equation*}
            and we recognize the \textit{Chebyshev semi-iterative method} described in \Cref{algo:cheby_method}.
        \end{proof}
        
        This method, unlike the Polyak heavy ball (PHB) method, uses a different step-size and momentum at each iteration. However, both are related, as taking the limit of $\omega_t$ as $t\rightarrow\infty$ in Algorithm \ref{algo:cheby_method} we obtain $\omega_{\infty} = 1 + m$ with $m = \left(\frac{1 - \sqrt{\kappa}}{1+\sqrt{\kappa}}\right)^2$. This correspond to the parameters of PHB.
        
        We note that this is only one way to construct a method that has the Chebsyshev polynomial as residual polynomial at every iteration. However, it is possible to construct a different update that have the Chebyshev polynomial at fixed iteration, see for instance \citep{young1953richardson, agarwal2021acceleration} for one such alterative that does not require momentum.
        
        \subsection{Polyak heavy ball method}\label{subsec:Polyak-heavy-ball-method}
        
        \begin{algorithm}[H]
            \caption{Polyak Heavy ball}\label{alg:HB1}
            \SetAlgoLined
            \textbf{Input:} $x_0$ \\
            \textbf{Set:} $m = \left(\frac{1 - \sqrt{\kappa}}{1+\sqrt{\kappa}}\right)^2$ and $h = \frac{2(1+m)}{L+\mu}$.
            ~\\
            $x_1 = x_0 - \frac{h}{1+m}\nabla f(x_0)$ \\
            \For{$t = 1, \dots$}{
                $x_{t+1} = x_t - h\nabla f(x_t) + m  (x_t - x_{t-1})$
            }
        \end{algorithm}
        
        \begin{Th}
            The iterates of the heavy ball algorithm verify
            \[
                x_t - x_* = P_t(H)(x_0 - x_*) \quad  \text{for all} \;\; t\in\mathbb{N},
            \]
            with $P_t$ defined as
            \begin{equation}
                P_t(\lambda) \triangleq \left(\sqrt{m}\right)^t \left[\frac{2m}{1+m}T_t(\sigma_1^{[\mu,L]}(\lambda)) + \frac{1-m}{1+m}U_t(\sigma_1^{[\mu,L]}(\lambda))\right]\,.
            \end{equation}
            Furthermore, this method enjoys a worst-case rate of the form
            \begin{equation}
                \|x_t-x_*\| = O\left(t\left( \frac{\sqrt{\kappa} - 1}{\sqrt{\kappa} + 1} \right)^t\right).
            \end{equation}
        \end{Th}
        
        \begin{proof} 
            From the update defined in \Cref{alg:HB1}, we identify
            \begin{align*}
                P_0(\lambda) & = 1 \\
                P_1(\lambda) & = 1 - \frac{h}{1+m}\lambda \\
                P_{t+1}(\lambda) & = (1+m-h\lambda)P_t(\lambda) - m P_{t-1}(\lambda).
            \end{align*}
            Introducing $\Tilde{P}_t \triangleq \frac{P_t}{\left(\sqrt{m}\right)^t}$, we have
            \begin{align*}
                \Tilde{P}_0(\lambda) & = 1 \\
                \Tilde{P}_1(\lambda) & = \frac{1 + m - h\lambda}{(1+m)\sqrt{m}} = \frac{2}{1+m}\sigma_1^{[\mu,L]}(\lambda) \\
                \Tilde{P}_{t+1}(\lambda) & = \frac{(1+m-h\lambda)}{\sqrt{m}}\Tilde{P_t}(\lambda) - \Tilde{P}_{t-1}(\lambda) \\
                & = 2\sigma_1^{[\mu,L]}(\lambda)\Tilde{P}_t(\lambda) - \Tilde{P}_{t-1}(\lambda).
            \end{align*}
            This is a second order recurrence, with 2 initializations. It allows us to identify uniquely the family
            \begin{equation}
                \Tilde{P}_t(\lambda) = \frac{2m}{1+m}T_t(\sigma_1^{[\mu,L]}(\lambda)) + \frac{1-m}{1+m}U_t(\sigma_1^{[\mu,L]}(\lambda)).
            \end{equation}
            where $U_t$ denotes the Chebyshev polynomial of the second kind of degree $t$. While both $T_t$ and $U_t$ verify the same recursion as $\Tilde{P}_t$ and $T_0 = U_0 = \Tilde{P}_0 = 1$, the difference between $T$ and $U$ comes when $T_1(X) = X$ and $U_1(X) = 2X$. This is how $\Tilde{P}_t$ ends being a linear combination of the $T_t$ and $U_t$.
            Finally,
            \begin{equation}
                P_t(\lambda) = \left(\sqrt{m}\right)^t \left[\frac{2m}{1+m}T_t(\sigma_1^{[\mu,L]}(\lambda)) + \frac{1-m}{1+m}U_t(\sigma_1^{[\mu,L]}(\lambda))\right].
            \end{equation}
            Since by definition $\sigma_1^{[\mu,L]}([\mu, L]) = [-1, 1]$, $T_t(\sigma_1^{[\mu,L]}(\lambda)) \leq 1$ and $U_t(\sigma_1^{[\mu,L]}(\lambda)) \leq t+1, \forall t\in\mathbb{N}$.
            Hence, $\forall \lambda \in [\mu, L]$,
            \begin{equation}
                P_t(\lambda) \leq \left(\sqrt{m}\right)^t \left[1 + \frac{1-m}{1+m} t \right] \leq (2\sqrt{\kappa}t+1) \left(\frac{1 - \sqrt{\kappa}}{1+\sqrt{\kappa}}\right)^t
            \end{equation}
            and
            \begin{equation}
                \|x_t-x_*\| = O\left(t\left( \frac{\sqrt{\kappa} - 1}{\sqrt{\kappa} + 1} \right)^t\right).
            \end{equation}
        \end{proof}

\section{Minimax Polynomials and Equioscillation Property}\label{apx:minimax-polynomials-and-equioscillation-property}
    
    \Cref{apx:optimal-methods-for-strongly-convex-smooth-quadratic-objective} dealt with optimal methods when $\Lambda = [\mu, L]$.
    Those methods could be derived since the minimax polynomial (Definition \ref{def:optimally_small_polynomial}) $Z_t^{[\mu, L]}$ is known.
    
    In this section we consider the problem of finding minimax polynomials in a more general setting.
    We provide a characterization of the minimax polynomial defined in definition \ref{def:optimally_small_polynomial}.
    For the sake of simplicity, we actually focus on the polynomial $\sigma_t^{\Lambda}$ solution of \eqref{eq:optimal_sigma}.
    We can easily adapt the result to $Z_t^{\Lambda}$ leveraging Remark \ref{rem:relationship_btw_optimal_and_minimax}.
    We prove the following theorem.
    \begin{Th}
        \label{th:equioscillation}
        Let $P_K$ be a degree $K$ polynomial verifying 
        $P_K(\Lambda) \subset [-1, 1]$.
        Then $P_K$ is the unique solution $\sigma_K^{\Lambda}$ of \cref{eq:optimal_sigma} if and only if there exists a sorted family $(\lambda_i)_{i \in \llbracket 0, K \rrbracket} \in \left(\overline{\Lambda}\right)^{K+1}$ (where $\overline{\Lambda}$ is the closure of $\Lambda$) such that $\forall i \in \llbracket 0, K \rrbracket, P_K(\lambda_i) = (-1)^i$.
    \end{Th}
    
    The following proof is technical and requires to introduce several new notations. Hence we first briefly describe the intuition before giving the actual complete proof.
    
    ($\Longleftarrow$): Assume $P_K$ ``oscillates''  $K+1$ times between $1$ and $-1$. Since $P_K$ has a degree $K$, it is completely determined by its values on those $K+1$ points, using the Lagrange interpolation representation. We prove that $P_K$ is optimal because any other polynomial $Q_K$, having different  values on those $K+1$ points would achieve a smaller value $Q_K(0)$ at 0.
    
    ($\Longrightarrow$): We prove this by contradiction. We assume that $P_K$ doesn't oscillate $K+1$ times between 1 and $-1$, and  prove that $P_K(0)$ is not optimal. To do so, we build a small perturbation $\varepsilon Q_K$ such that $P_K + \varepsilon Q_K$ is a polynomial of degree $K$, which values on $\Lambda$ are all in $[-1;1]$, and with an higher value at 0.
    
    (Uniqueness) We reuse the Lagrange interpolation representation to justify that 2 optimal polynomials must ``oscillate'' on the same points, therefore are equal.
    
    \begin{proof}
        We prove successively both directions:
        
        ($\Longleftarrow$): \textit{Assume $\exists \lambda_0 < \lambda_1 < \cdots < \lambda_K$ such that
        \begin{equation}
            \forall i \in \llbracket 0, K \rrbracket, P_K(\lambda_i) = (-1)^i \quad\mathrm{and}\quad P_K(\Lambda) \subset [-1, 1].
        \end{equation}
        We aim to prove that $P_K$ is  the unique solution $\sigma_K^{\Lambda}$ of \cref{eq:optimal_sigma}, that is for any other polynomial $Q_K$ of degree $K$ verifying $Q_K(\Lambda) \subset [-1, 1]$, $P_K(0) \geq Q_K(0)$.
        }
        
        We introduce such a polynomial  $Q_K$ of degree $K$ and  bounded in absolute value by $1$ on $\Lambda$. Let's define,  for all $i \in \llbracket 0,\, K \rrbracket$, 
        \begin{equation}
           v_i \triangleq Q_K(\lambda_i) \in [-1, 1].
        \end{equation}
        These  $K+1$ values characterize  $Q_K$ (of degree $K$), and we can decompose it over Lagrange interpolation polynomials.
        We have
        \begin{equation}
            Q_K = \sum_{i=0}^K v_i L_{\lambda_i} \quad\mathrm{where}\quad L_{\lambda_i}(X) \triangleq \prod_{j \neq i}\frac{X - \lambda_j}{\lambda_i - \lambda_j}\ .
        \end{equation}
        The value at 0 can be computed as
        \begin{equation}
            Q_K(0) = \sum_{i=0}^K v_i L_{\lambda_i}(0) = \sum_{i=0}^K v_i \prod_{j \neq i}\frac{\lambda_j}{\lambda_j - \lambda_i}\ .
        \end{equation}
        Maximizing this linear function of $(v_i)_{i \in\llbracket 0,\, K \rrbracket}$  over the $\ell_\infty$ ball $B_\infty(1)\triangleq \lbrace  (v_i)_{i \in\llbracket 0,\, K \rrbracket} , \forall i, -1 \leq v_i \leq 1\rbrace$  leads to, for $v^* \triangleq \argmin_{v\in B_\infty(1)} \sum_{i=0}^K v_i \prod_{j \neq i}\frac{\lambda_j}{\lambda_j - \lambda_i}\ $, 
        \begin{equation}
            v_i^* = \mathrm{sgn}\left(\prod_{j \neq i}\frac{\lambda_j}{\lambda_j - \lambda_i}\right) = (-1)^{i}.
        \end{equation}
        where $\mathrm{sgn}$ is the sign function (which maps 0 to 0, $\mathbb{R}_{<0}$ to $-1$, and $\mathbb{R}_{>0}$ to 1).
        Finally,
        \begin{equation}
            P_K(0) \geq Q_K(0)
        \end{equation}
        which concludes the proof.
        
        ($\Longrightarrow$): \textit{Assume $P_K$ alternates $s<K+1$ times  between $-1$ and $1$ on $\overline{\Lambda}$. We want to show that $P_K$ is not optimal in the sense described above. To do so, we construct a perturbation of $P_K$ that increases its value in 0 while still satisfying the constraint $ P_K(\Lambda) \subset [-1, 1]$.
        }
        
        Let's define 
        \begin{equation}
            \lambda_0^{(1)} < \cdots < \lambda_0^{(\nu_0)} < \lambda_1^{(1)} < \cdots < \lambda_1^{(\nu_1)} < \cdots < \lambda_{s-1}^{(1)} < \cdots < \lambda_{s-1}^{(\nu_{s-1})}
        \end{equation}
        such that
        \begin{equation}
            P_K(\lambda_i^{(j)}) = (-1)^i \quad\mathrm{and}\quad \forall \lambda \in \overline{\Lambda}, \left(\exists (i, j) | \lambda = \lambda_i^{(j)} \mathrm{~or~} |P_K(\lambda)|<1 \right).
        \end{equation}
        In short, $\left(\lambda_i^{(j)}\right)_{(i, j)}$ describes all the extremal points of $P_K$ in $\Lambda$. The indices change when the sign changes, while the exponents are used to express the possible consecutive repetitions of the same value ($-1$ or $1$).
        
        Set $(r_i)_{i\in\llbracket 0, s\rrbracket}$ as any set of positive numbers satisfying:
        \begin{equation}
            0 < r_0 < \inf(\Lambda) < \lambda_0^{(1)} < \lambda_0^{(\nu_0)} < r_1 < \cdots < r_s < \lambda_{s-1}^{(1)} < \lambda_{s-1}^{(\nu_{s-1})} < \sup(\Lambda) < r_{s}.
        \end{equation}
        By definition, each interval $[r_i, r_{i+1}]$, $i\in \llbracket 0,\, s-1 \rrbracket$, contains $\lambda_i^{(j)}$ for all $j$, but no other extremal points of $P_K$ in $\overline{\Lambda}$.
        Hence, $P_K([r_i, r_{i+1}] \cap \overline{\Lambda})$ doesn't contain $(-1)^{i+1}$.
        Since, $\bigcup_{i<s, i \mathrm{~even}} [r_i, r_{i+1}]\cap\overline{\Lambda}$ is compact, and by continuity of $P_K$, $P_K\left(\bigcup_{i<s, i \mathrm{~even}} [r_i, r_{i+1}]\cap\overline{\Lambda}\right)$ is compact.
        Therefore,
        \begin{equation}
            \exists \varepsilon_{-1}>0 | P_K\left(\bigcup_{i<s, i \mathrm{~even}} [r_i, r_{i+1}]\cap\overline{\Lambda}\right) \subset [-1 + \varepsilon_{-1}, 1].
        \end{equation}
        Similarly, we obtain
        \begin{equation}
            \exists \varepsilon_{1}>0 | P_K\left(\bigcup_{i<s, i \mathrm{~odd}} [r_i, r_{i+1}]\cap\overline{\Lambda}\right) \subset [-1, 1 - \varepsilon_{1}].
        \end{equation}
        
        We are now equipped to build the aforementioned perturbation. Let \begin{equation}
            Q_K(X) \triangleq \prod_{i\in\llbracket 0, s-1\rrbracket} (r_i - X).
        \end{equation}
        Note that $Q_K$ has a degree $s \leq K$ and satisfies
        \begin{equation}
            Q_K\left(\bigcup_{i<s, i \mathrm{~even}} [r_i, r_{i+1}]\cap\overline{\Lambda}\right) \subset \mathbb{R}^- \quad \mathrm{and} \quad Q_K\left(\bigcup_{i<s, i \mathrm{~odd}} [r_i, r_{i+1}]\cap\overline{\Lambda}\right) \subset \mathbb{R}^+.
        \end{equation}
        
        Moreover, those sets are compact , by continuity of $Q_K$, and consequently bounded. We can therefore choose a small enough $\varepsilon>0$ such that
        \begin{equation*}
            \hspace{-1cm}
            \varepsilon \min Q_K\left(\bigcup_{i<s, i \mathrm{~even}} [r_i, r_{i+1}]\cap\overline{\Lambda}\right) > -\varepsilon_{-1}
            \quad \mathrm{and} \quad
            \varepsilon \max Q_K\left(\bigcup_{i<s, i \mathrm{~odd}} [r_i, r_{i+1}]\cap\overline{\Lambda}\right) < \varepsilon_{1}.
        \end{equation*}
        This leads to 
        \begin{equation}
            (P_K + \varepsilon Q_K)(\Lambda) \subset [-1, 1].
        \end{equation}
        And as by definition, $Q_K(0)>0$, 
        \begin{equation}
            (P_K + \varepsilon Q_K)(0) > P_K(0).
        \end{equation}
        Finally $(P_K + \varepsilon Q_K) \in \mathbb{R}_K[X]$.
This  proves that $P_K$ is not optimal.
        
        (Uniqueness) \textit{Here, we prove  that the optimal polynomial is necessarily unique. To do so, we introduce 2 optimal polynomials and show there must actually be identical.}
        
        Let $P_K$ an optimal polynomial and $(\lambda_i)_{i \in \llbracket 0, K \rrbracket} \in \Lambda^{K+1}$ a family on which $P_K$ interpolates alternatively $1$ and $-1$. Let any other feasible polynomial $Q_K$ and $(v_i)_{i \in \llbracket 0, K \rrbracket}$ its values on $(\lambda_i)_{i \in \llbracket 0, K \rrbracket}$:
        \begin{equation}
            Q_K = \sum_{i=0}^K v_i L_{\lambda_i}.
        \end{equation}
        We have showed in the first point of this proof that the optimal values of $v_i$ are alternatively $1$ and $-1$. Consequently, if $Q_K$ is also optimal,
        \begin{equation}
            Q_K(\lambda_i) = P_K(\lambda_i)
        \end{equation}
        for all $i \in \llbracket 0, K \rrbracket$, which characterizes polynomials of degree $K$. Then
        \begin{equation}
            Q_K = P_K
        \end{equation}
        which shows that the optimal polynomial is unique.
    \end{proof}
    
    We now give the formal statement and the proof of the second result, used in \Cref{subsec:optim}. 
    \begin{Th}\label{thm:optimality_from_equioscillation}
        $T_n(\sigma_K)$ is optimal for all $n$ if and only if $\sigma_K$ verifies the equioscillation property (Definition \ref{def:equiosc}, hence $\sigma_K = \sigma_K^{\Lambda}$ by Theorem \ref{th:equioscillation}) and $\overline{\Lambda} = \sigma_K^{-1}([-1, 1])$, i.e. the inverse mapping $\sigma_K^{-1}$ transforms the interval $[-1,1]$ into exactly $\overline{\Lambda}$.
    \end{Th}
    
    Before providing the proof, we first highlight that the property
    \begin{equation}
        \forall \lambda\in\Lambda, \sigma_K(\lambda) \in [-1, 1]
    \end{equation}
    can equivalently be written
    \begin{equation}
        \overline{\Lambda} \subset \sigma_K^{-1}([-1, 1]).
    \end{equation}
    In other words, we are interested in the case where the reverse inclusion holds as well. This means that
     \begin{equation}
          \sigma_K(\lambda) \in [-1, 1]\Rightarrow  \lambda\in\overline{\Lambda}.
    \end{equation}
    This corresponds to a stronger form of optimality of $\sigma_K$: it ``fully'' uses the available assumption related to $\Lambda$, in the sense that no point can be added to $\overline \Lambda$ without breaking the condition $\sigma_K(\Lambda)\subset [-1;1]$. For example, on \Cref{fig:minimax_polynomials}, $\sigma_3^\Lambda$ \textit{does not} satisfy the later property on the center graph, but satisfies it on the right  graph.
    Here, we show that under this condition, $T_n(\sigma_K) = T_n(\sigma_K^{\Lambda})$ is optimal (in the sense of \eqref{eq:optimal_sigma}) for all $n\in\mathbb{N}$.
    
    In Section~\ref{subsec:optim}, we give another view of this condition for $T_n(\sigma_K)$ to be optimal for all $n$.
    We can decompose $\Lambda$ as the union of $K$ intervals $\Lambda_i$ such that they have disjoint interiors and they are all mapped to $[-1, 1]$ by $\sigma_K$. Hence, $\sigma_K$ maps $\Lambda$ to $[-1, 1]$ exactly $K$ times.
    
    \begin{proof}
        From Theorem \ref{th:equioscillation}, $T_n(\sigma_K)$ is optimal for all $n$ if and only if, for all $n$, there exist a sorted family of $(\lambda_i)_{i\in\llbracket 0, nK \rrbracket}$ such that, $T_n(\sigma_K(\lambda_i)) = (-1)^i$.
        
        Let $n \in \mathbb{N}$. We observe that  by definition of $T_n$,  
        \begin{equation}
            T_n(\sigma_K(\lambda)) = \pm 1 \quad \mathrm{if~and~only~if} \quad \exists j\in \llbracket 0, n\rrbracket | \sigma_K(\lambda) = \cos{\frac{j\pi}{n}}.
            \label{eq:cos_values_reaching_pm1_through_Tn}
        \end{equation}
        
        We successively treat both directions:
        ($\Longleftarrow$) \textit{we assume $\sigma_K$ oscillates and $\overline{\Lambda} = \sigma_K^{-1}([-1, 1])$. We aim to  prove that $T_n(\sigma_K)$ is optimal for all $n\in\mathbb{N}$.
        }
            
        By equioscillation property, we know that there exists $\lambda_i'$ such that
        \begin{equation}
            \sigma_K(\lambda_i') = (-1)^i.
        \end{equation}
        
        By the intermediate value theorem, we know that for any $i\in \llbracket 0;K\rrbracket$,  between the pair $\lambda_i', \lambda_{i+1}'$, there exist sorted $(\mu_i^j)_{ni<j<(n+1)i}$ such that for all $j \in \llbracket ni+1;(n+1)i-1\rrbracket$,
        \begin{equation}
            \sigma_K(\mu_i^j) = \cos{\frac{j\pi}{n}}.
        \end{equation}
        We  identify $\lambda_{ni} = \lambda_i'$ and $\lambda_{j} = \mu_{\lfloor j/n\rfloor}^j$ for all $j$ not multiple of $n$.
    Then, for all $\ell \in\llbracket 0, nK\rrbracket$:
        \begin{equation}
            T_n(\sigma_K(\lambda_\ell)) = (-1)^\ell.
        \end{equation}
        By Theorem \ref{th:equioscillation}, we conclude that  $T_n(\sigma_K)$ is optimal for all $n\in\mathbb{N}$.
            
        ($\Longrightarrow$) \textit{We assume $T_n(\sigma_K)$ is optimal for all $n\in\mathbb{N}$. Clearly, $\sigma_K$ is optimal ($n=1$), and then equioscillates. We prove that moreover
        \begin{equation}
            \overline{\Lambda} = \sigma_K^{-1}([-1, 1]).
        \end{equation}
        }
        On the one hand, for any $j \in \llbracket 0, n\rrbracket$, there exist at most $K$ different $\lambda$ that verifies $\sigma_K(\lambda) = \cos{\frac{j\pi}{n}}$ since $\sigma_K$ has a degree $K$ and is not constant.
       Therefore, there exist at most $(n+1)K$ different $\lambda$ such that $\exists j\in \llbracket 0, n\rrbracket | \sigma_K(\lambda) = \cos{\frac{j\pi}{n}}$, and by Eq.\eqref{eq:cos_values_reaching_pm1_through_Tn}, there thus exist at most $(n+1)K$ different $\lambda$ such that $T_n(\sigma_K(\lambda) = \pm 1$.
        
        On the other hand, the optimality of $T_n(\sigma_K)$ implies the existence of  \textit{at least} $nK+1$ such $\lambda$ in $\overline \Lambda$.
        
        Hence  all but at most $K-1$ values $\lambda$ such that $\sigma_K(\lambda) \in \lbrace \cos{\frac{j\pi}{n}}, j\in  \llbracket 0, n\rrbracket \rbrace$  belong to $\overline{\Lambda}$.
        
        This holds for all $n$. Therefore for $n$ large enough, all $x$ such that $\sigma(x)\in [-1, 1]$ are as close as we want to some $\lambda\in\overline{\Lambda}$.
        Since $\overline{\Lambda}$ is a closed set, then all $x$ such that $\sigma(x)\in [-1, 1]$ are actually in $\overline{\Lambda}$.
        
        We conclude
        \begin{equation}
            \overline{\Lambda} \supset \sigma_K^{-1}([-1, 1]).
        \end{equation}
        
        \end{proof}

\section{Cyclical step-sizes}\label{apx:cycling}
    
    In this appendix, we provide an analysis of momentum methods with cyclical step-sizes and derive some non-asymptotically optimal variants.
    
    \subsection{Derivation of optimal algorithm with \texorpdfstring{$K=2$}{K=2} alternating step-sizes}\label{apx:derivation-of-optimal-algorithm-withstep-sizes}
    
        In this section, we consider the case where $\Lambda$ is the union of 2 intervals of same size, as described in Section \ref{sec:alternating_hb}.
        
        We start by introducing the following algorithm, and we will prove later that this algorithm is optimal (Theorem \ref{thm:cyclical_chebyshev_2_steps})
        \begin{algorithm}[ht] \label{algo:cyclical_chebyshev_2_steps}
            \caption{Optimal momentum method with alternating step-sizes ($K=2$)}
            \SetAlgoLined
            \textbf{Input:} Initialization $x_0$, $\muone < \Lone < \mutwo < \Ltwo$ $\quad$ (where $\Lone-\muone = \Ltwo-\mutwo$)\\
            \textbf{Set:} $\rho = \frac{\Ltwo+\muone}{\Ltwo-\muone}$, $R=\frac{\textcolor[HTML]{4DAF4A}{\boldsymbol{\mu}_2 - \boldsymbol{L}_1}}{\textcolor[HTML]{984ea3}{\boldsymbol{L}_2 - \boldsymbol{\mu}_1}}$, $c = \sqrt{\frac{\rho^2 - R^2}{1 - R^2}}$ \\
            ~\\
            $\omega_0 = 2$ \\
            $x_1 = x_0 - \frac{1}{\Lone}\nabla f(x_0)$ \\
            \For{$t = 1,\, 2,\,\ldots$}{
                \begin{align*}
                    \omega_{t} & ~ = \left(1 - \frac{1}{4c^2}\omega_{t-1}\right)^{-1} \\
                    h_t & = \textstyle \frac{\omega_t}{\Lone} \quad \text{(if $t$ is even)}, \qquad h_t = \frac{\omega_t}{\mutwo} \quad \text{(if $t$ is odd)} \\
                    x_{t+1} & = ~ x_t - h_t \nabla f(x_t) + (\omega_t-1)(x_{t}-x_{t-1})
                \end{align*}
                \vspace{-2ex}
            }
        \end{algorithm}
        
        \begin{Th}\label{thm:cyclical_chebyshev_2_steps}
            Let $f\in \mathcal{C}_\Lambda$ and $x_0\in\mathbb{R}^d$. Assume $\Lambda$ defined as in \eqref{eq:def_lambda}. The iterates of Algorithm ~\ref{algo:cyclical_chebyshev_2_steps} verifies the condition
            \begin{equation}\label{eq:best_2K_polynomial}
                x_{2n} - x_* = \frac{T_n(\sigma_2^{\Lambda}(H))}{T_n(\sigma_2^{\Lambda}(0))} (x_0 - x_*)
            \end{equation}
            and this is the optimal convergence rate over $\mathcal{C}_\Lambda$.
        \end{Th}
        
        \begin{proof}
            
            We begin by showing the optimality of the algorithm. Using Proposition~\ref{prop:2interval_opt_iif_same_size}, the polynomial in \eqref{eq:best_2K_polynomial} equioscillates on $\Lambda$, which makes it optimal by Theorem \ref{th:equioscillation}. By optimal, this means this is the optimal convergence rate any first order algorithm can reach (See~\eqref{eq:rate_convergence_polynomial}). We invite the reader to read Appendix~\ref{apx:case-of-a-2-step-sizes-cycle}, where we study in details the properties of the alternating steps sizes strategy (i.e., $K=2$).

            As in Appendix~\ref{subsec:chebyshev-method}, we derive here the constructive approach that leads us to this algorithm.

            We now start showing that the iterates of Algorithm \ref{algo:cyclical_chebyshev_2_steps} follow \eqref{eq:best_2K_polynomial}. From \cref{eq:best_2K_polynomial}, projecting onto the eigenspace of eigenvalue $\lambda$,
            \begin{equation}
                x_{2n} - x_* = \frac{T_n(\sigma_2^{\Lambda}(\lambda))}{T_n(\sigma_2^{\Lambda}(0))} (x_0 - x_*).
            \end{equation}
            Then, we find a recursion definition for the subsequence $(x_{2n})_{n\in\mathbb{N}}$.
            Let $n \geq 1$.
            \begin{align}
                x_{2(n+1)} - x_* &~ = \frac{T_{n+1}(\sigma_2^{\Lambda}(\lambda))}{T_{n+1}(\sigma_2^{\Lambda}(0))} (x_0 - x_*), \\
                &~ = \frac{2 \, \sigma_2^{\Lambda}(\lambda) \, T_n(\sigma_2^{\Lambda}(\lambda)) - T_{n-1}(\sigma_2^{\Lambda}(\lambda))}{T_{n+1}(\sigma_2^{\Lambda}(0))} (x_0 - x_*), \\
                &~ = \frac{2 \, \sigma_2^{\Lambda}(\lambda) \, T_n(\sigma_2^{\Lambda}(0))}{T_{n+1}(\sigma_2^{\Lambda}(0))} (x_{2n} - x_*) - \frac{T_{n-1}(\sigma_2^{\Lambda}(0))}{T_{n+1}(\sigma_2^{\Lambda}(0))} (x_{2(n-1)} - x_*).
            \end{align}
            Note that if $\sigma_2^{\Lambda}(\lambda)$ were a degree 1 polynomial in $\lambda$, then we would recognize a momentum update. Here, $\sigma_2^{\Lambda}(\lambda)$ is actually a degree $2$ polynomial in $\lambda$. We will then try to identify $2$ steps of momentum. From here, let
            \begin{equation}
                c \triangleq \frac{1}{2}\left(\left(\sigma_K(0) + \sqrt{\sigma_K(0)^2-1}\right)^{1/2} + \left(\sigma_K(0) - \sqrt{\sigma_K(0)^2-1}\right)^{1/2}\right) = \sqrt{\frac{\sigma_K(0)+1}{2}}
            \end{equation}
            be the unique positive real number $c$ verifying $T_2(c) = 2c^2-1 = \sigma_K(0)$. We end up with
            \begin{equation}\label{eq:2_momentum_steps_equivalent}
                x_{2(n+1)} - x_* = \frac{2 \, \sigma_2^{\Lambda}(\lambda) \, T_{2n}(c)}{T_{2(n+1)}(c)} (x_{2n} - x_*) - \frac{T_{2(n-1)}(c)}{T_{2(n+1)}(c)} (x_{2(n-1)} - x_*).
            \end{equation}
            Note, the above equation suggests to introduce the sequence $z_l \triangleq T_l(c)(x_l - x_*)$. Indeed, the above equality simplifies
            \begin{equation}\label{eq:2_steps_cheby_in_z}
                z_{2(n+1)} = 2 \, \sigma_2^{\Lambda}(\lambda) \, z_{2n} - z_{2(n-1)}.
            \end{equation}
            Let's look for two steps of the cyclical heavy ball method that are together equivalent to \eqref{eq:2_momentum_steps_equivalent}.
            We look for an algorithm of the form
            \begin{equation}
                \forall n \geq 0,
                    x_{n+1} = x_{n} - h_{n} \nabla f(x_{n}) + \frac{T_{n-1}(c)}{T_{n+1}(c)} \left( x_{n} - x_{n-1} \right),
            \end{equation}
            i.e, projecting again onto the eigenspace of eigenvalue $\lambda$, we obtain
            \begin{equation}\label{eq:one_step_chebyshev_decomposition}
                \forall n \geq 0,
                    x_{n+1}-x_* = \left(1 + \frac{T_{n-1}(c)}{T_{n+1}(c)} - h_{n} \lambda \right) (x_{n}-x_*) - \frac{T_{n-1}(c)}{T_{n+1}(c)} \left( x_{n-1}-x_* \right).
            \end{equation}
            Here we introduce the notation
            \begin{equation}
                \omega_{l} \triangleq \left(1 + \frac{T_{l-1}(c)}{T_{l+1}(c)}\right) = 2c\frac{T_l(c)}{T_{l+1}(c)},
            \end{equation}
            and the change of variable
            \begin{equation}
                \Tilde{h_l} \triangleq \frac{h_l}{\omega_{l}}.
            \end{equation}
            We rewrite \eqref{eq:one_step_chebyshev_decomposition} in terms of the sequence $z$ and using the sequence $\Tilde{h}$,
            \begin{align}
                \forall n \geq 0,
                    z_{n+1} & ~ = T_{n+1}(c) \left(1 + \frac{T_{n-1}(c)}{T_{n+1}(c)} - h_{n} \lambda \right) (x_{n}-x_*) - z_{n-1} \\
                    & ~ = \left(2cT_n(c) (1 - \Tilde{h}_{n} \lambda)\right) (x_n - x_*) - z_{n-1} \\
                    & ~ = \left(2c (1 - \Tilde{h}_{n} \lambda)\right) z_n - z_{n-1}.
            \end{align}
            We now need to find the right sequence $\Tilde{h}_n$ such that we recover \cref{eq:2_steps_cheby_in_z}.
            Combining the 2 following
            \begin{align}
                z_{2n+1} & ~ = \left(2c (1 - \Tilde{h}_{2n} \lambda)\right) z_{2n} - z_{2n-1} \\
                z_{2n+2} & ~ = \left(2c (1 - \Tilde{h}_{2n+1} \lambda)\right) z_{2n+1} - z_{2n}
            \end{align}
            by isolating the odd index in the second equation and plugging it in the first one, we get
            \begin{equation}
                z_{2n+2} = \left(4c^2 (1 - \Tilde{h}_{2n} \lambda)(1 - \Tilde{h}_{2n+1} \lambda) - 1 - \frac{2c (1 - \Tilde{h}_{2n+1} \lambda)}{2c (1 - \Tilde{h}_{2n-1} \lambda)}\right) z_{2n} - \frac{2c (1 - \Tilde{h}_{2n+1} \lambda)}{2c (1 - \Tilde{h}_{2n-1} \lambda)} z_{2n-2}.
            \end{equation}
            We need to identify
            \begin{align}
                2\sigma_2^{\Lambda}(\lambda) & ~ = 4c^2 (1 - \Tilde{h}_{2n} \lambda)(1 - \Tilde{h}_{2n+1} \lambda) - 1 - \frac{2c (1 - \Tilde{h}_{2n+1} \lambda)}{2c (1 - \Tilde{h}_{2n-1} \lambda)}\,, \\
                1 & ~ = \frac{2c (1 - \Tilde{h}_{2n+1} \lambda)}{2c (1 - \Tilde{h}_{2n-1} \lambda)}.
            \end{align}
            Hence, we conclude from the second equation that $\Tilde{h}_{2n+1} = \Tilde{h}_{2n-1} = \Tilde{h}_1$ is independent of $n$.
            And the first equation then becomes
            \begin{equation}
                2\sigma_2^{\Lambda}(\lambda) = 4c^2 (1 - \Tilde{h}_{2n} \lambda)(1 - \Tilde{h}_1 \lambda) - 2
            \end{equation}
            leading also to $\Tilde{h}_{2n}$ independent of $n$.
            We observe an alternating strategy of the ``pseudo-step-sizes'' $\Tilde{h}_0$ and $\Tilde{h}_1$.
            Finally, we must fix them to
            \begin{equation}
                \sigma_2^{\Lambda}(\lambda) = 2c^2 (1 - \Tilde{h}_0 \lambda)(1 - \Tilde{h}_1 \lambda) - 1.
            \end{equation}
            Note this is possible because the equation above is valid for $\lambda=0$ for any choice of $\Tilde{h}_0$ and $\Tilde{h}_1$ and the polynomial $\sigma_2^{\Lambda}+1$ can be defined by its value in 0 and its roots that are exactly $\frac{1}{\Tilde{h}_0}$ and $\frac{1}{\Tilde{h}_1}$. And from~\eqref{eq:sigma_2_reaches_m1}, those values are $\mutwo$ and $\Lone$, which gives the values $\Tilde{h}_0 = \frac{1}{\Lone}$ and $\Tilde{h}_1=\frac{1}{\mutwo}$.
            
            We now sum up what we have so far.
            Setting $c$, $\Tilde{h}_0$ and $\Tilde{h}_1$ as described above, the iterations
            \begin{equation}
                \forall n \geq 1,
                    x_{n+1} = x_{n} - \left(1 + \frac{T_{n-1}(c)}{T_{n+1}(c)}\right)\Tilde{h}_{\text{mod}(n,2)} \nabla f(x_{n}) + \frac{T_{n-1}(c)}{T_{n+1}(c)} \left( x_{n} - x_{n-1} \right)
            \end{equation}
            lead to the recursion \eqref{eq:2_steps_cheby_in_z}.
            
            Let define $x_1 = x_0 - \Tilde{h}_0 \nabla f(x_0)$, and from the above
            \begin{align}
                x_2 & ~ = x_1 - \left(1+\frac{1}{2c^2-1}\right) \Tilde{h}_1\lambda (x_1 - x_*) + \frac{1}{2c^2-1} (x_1 - x_0) \\
                x_2 - x_* & ~ = \frac{2c^2}{\sigma_2^{\Lambda}(0)} \left(1 - \Tilde{h}_1\lambda \right) (x_1 - x_*) - \frac{1}{\sigma_2^{\Lambda}(0)} (x_0-x_*) \\
                & ~ = \frac{2c^2}{\sigma_2^{\Lambda}(0)} \left(1 - \Tilde{h}_1\lambda \right) \left(1 - \Tilde{h}_0 \lambda \right)(x_0 - x_*) - \frac{1}{\sigma_2^{\Lambda}(0)} (x_0-x_*) \\
                & ~ = \frac{\sigma_2^{\Lambda}(\lambda)}{\sigma_2^{\Lambda}(0)}(x_0 - x_*) \\
                z_2 & ~ = \sigma_2^{\Lambda}(\lambda).
            \end{align}
            
            Finally, the sequence $z_{2n}$ is defined by
            \begin{align}
                z_0 & ~ = 1, \\
                z_1 & ~ = \sigma_2^{\Lambda}(\lambda), \\
                z_{2(n+1)} & ~ = 2 \, \sigma_2^{\Lambda}(\lambda) \, z_{2n} - z_{2(n-1)}.
            \end{align}
            which defines exactly $T_n(\sigma_2^{\Lambda}(\lambda))$.
            We conclude $x_{2n}-x_* = \frac{T_n(\sigma_2^{\Lambda}(\lambda))}{T_n(\sigma_2^{\Lambda}(0))}(x_0 - x_*)$.
            
            We sum up the algorithm used to reach the above equality:
            \begin{align}
                x_1 & ~ = x_0 - \Tilde{h}_0 \nabla f(x_0), \\
                \forall n \geq 0,
                x_{n+1} & ~ = x_{n} - \omega_n \Tilde{h}_{n} \nabla f(x_{n}) + (\omega_n - 1) \left( x_{n} - x_{n-1} \right)\,.
            \end{align}
            with $\omega_{n} = \left(1 + \frac{T_{n-1}(c)}{T_{n+1}(c)}\right) = \frac{2c T_{n}(c)}{T_{n+1}(c)}$.
            Note the recursion
            \begin{align}
                \omega_{n}^{-1} & ~ = \frac{T_{n+1}(c)}{2c T_{n}(c)} \\
                & ~ = \frac{2c T_{n}(c) - T_{n-1}(c)}{2c T_{n}(c)} \\
                & ~ = 1 - \frac{T_{n-1}(c)}{2c T_{n}(c)} \\
                & ~ = 1 - \frac{1}{4c^2}\frac{2c T_{n-1}(c)}{ T_{n}(c)} \\
                & ~ = 1 - \frac{1}{4c^2}\omega_{n-1}.
            \end{align}
            Finally, the sequence $\omega$ can be computed online using the recursion
            \begin{equation}
                \omega_{n} = \frac{1}{1 - \frac{1}{4c^2}\omega_{n-1}}
            \end{equation}
            with $\omega_0 = 2$.
        \end{proof}
        
        In this appendix, as well as in Appendix~\ref{apx:optimal-methods-for-strongly-convex-smooth-quadratic-objective}, we end up with some equality of the form
        \begin{equation}
            \|x_t - x_*\| = \frac{T_n(\sigma_K(H))}{T_n(\sigma_K(0))} \|x_0 - x_*\|\,.
        \end{equation}
        The next theorem explains how to derive the rate factor from it.
        
        \rateconvergencesigma*
        
        \begin{proof}
            We observe that the rate factor of the method is upper bounded by
            \begin{equation}
                \sqrt[t]{\underset{\lambda \in \Lambda}{\sup} |Z_t^{\Lambda}(\lambda)|} = \sqrt[t]{\underset{\lambda \in \Lambda}{\sup}  \left|  \frac{T_{t/K}\left( \sigma_K^{\Lambda}(\lambda) \right)}{T_{t/K}\left(\sigma_0\right)} \right|} = \sqrt[t]{ \frac{1}{|T_{t/K}\left(\sigma_0\right)|}} \quad \text{if } \underset{\lambda \in \Lambda}{\sup} |\sigma_K(\lambda)| = 1.
            \end{equation}
            
            Since $ \sigma_0 >1$, and by using the explicit formula of Chebyshev polynomials, we have that
            \begin{equation}
                T_{t/K}\left(\sigma_0\right) = \frac{\left(\sigma_0 + \sqrt{\sigma_0^2-1}\right)^{t/K} + \left(\sigma_0 - \sqrt{\sigma_0^2-1}\right)^{t/K}}{2} \underset{t\rightarrow\infty}{\sim} \frac{\left(\sigma_0 + \sqrt{\sigma_0^2-1}\right)^{t/K}}{2}.
            \end{equation}
            Taking the limit gives
            \begin{equation}
                \lim\limits_{t\rightarrow\infty} \sqrt[t]{ \frac{1}{|T_{t/K}\left(\sigma_0\right)|}} = \left(\frac{1}{\sigma_0+\sqrt{\sigma_0^2-1}}\right)^{\frac{1}{K}} = \left(\sigma_0-\sqrt{\sigma_0^2-1}\right)^{\frac{1}{K}}.
            \end{equation}
        \end{proof}
    
    \subsection{Derivation of heavy ball with \texorpdfstring{$K$}{K} step-sizes cycle}\label{apx:derivation-of-heavy-ball-withstep-sizes-cycle}
    
        In this section, we consider heavy ball algorithm with a cycle of $K$ different step-sizes.
        For convenience, we restate Algorithm \ref{algo:cyclical_heavy_ball} below.

        \HBK*
        
        We first recall the convergence theorem~\ref{thm:general_rate_convergence} stated in Section~\ref{ssec:resulting_algo}.
        \generalrateconvergence*

        \begin{proof}
            Note a first trick. Let's define $x_{-1} \triangleq x_0 - \dfrac{h_0}{1 + m}\nabla f(x_0)$.
            This way, $x_{t+1} = x_t - h_{\text{mod}(t,K)} \nabla f(x_t) + m(x_{t}-x_{t-1})$ holds for any $t\geq0$ (including $t=0$).
            
            Now, let's introduce the polynomials $P_t$ defined by Proposition \ref{prop:link_algo_poly} as $x_t - x_* = P_t(H)(x_0 - x_*)$.
            From now, in order to highlight the $K$-cyclic behavior, we introduce the indexation $t=nK+r$, with $r\in\llbracket 0, K-1\rrbracket$.
            
            We verify the following:
            \begin{align}
                P_{-1}(\lambda) & ~ = 1 - \dfrac{h_0\lambda}{1+m}, \\
                P_0(\lambda) & ~ = 1, \\
                \forall n\geq0, r\in\llbracket 0, K-1 \rrbracket, ~~ P_{nK+r+1}(\lambda) &~ = (1+m - h_r\lambda) P_{nK+r}(\lambda) - m P_{nK+r-1}(\lambda).
            \end{align}
            
            In order to get rid of the last occurrence of $m$ in equation above, we introduce $\Tilde{P}_t(\lambda) \triangleq \frac{1}{\left(\sqrt{m}\right)^t}P_t(\lambda)$.
            
            This way, the above can be written
            \begin{align}
                \Tilde{P}_{-1}(\lambda) & ~ = \sqrt{m} \left(1 - \dfrac{h_0\lambda}{1+m}\right) = \dfrac{2m}{1+m}\sigma_0(\lambda), \\
                \Tilde{P}_0(\lambda) & ~ = 1, \\
                \forall n\geq0, r\in\llbracket 0, K-1 \rrbracket, ~~ \Tilde{P}_{nK+r+1}(\lambda) &~ = \dfrac{1+m - h_r\lambda}{\sqrt{m}} \Tilde{P}_{nK+r}(\lambda) - \Tilde{P}_{nK+r-1}(\lambda).
            \end{align}
            
            In the following, we want to determine a formulation for the polynomials $\Tilde{P}_{nK}$.
            In order to do so, we introduce the following operator:
            
            \begin{equation}\label{eq:matrix_equality}
                A(\lambda) \triangleq 
                \begin{pmatrix}
                    \dfrac{1+m - h_{K-1}\lambda}{\sqrt{m}} & -1 \\
                    1 & 0
                \end{pmatrix}
                \cdots
                \begin{pmatrix}
                    \dfrac{1+m - h_0\lambda}{\sqrt{m}} & -1 \\
                    1 & 0
                \end{pmatrix}
                \triangleq
                \begin{pmatrix}
                    a(\lambda) & b(\lambda) \\
                    c(\lambda) & d(\lambda)
                \end{pmatrix}
            \end{equation}
            
            as well as the scalar valued function
            
            \begin{equation}
                \sigma(\lambda;\{h_i\}, m) \triangleq \frac{1}{2}\mathrm{Tr}(A(\lambda))\,.\label{eq:sigma_def_via_trace}
            \end{equation}
            
            This operator comes naturally in
            
            \begin{align}
                \begin{pmatrix}
                    \Tilde{P}_{(n+1)K}(\lambda) \\
                    \Tilde{P}_{(n+1)K-1}(\lambda)
                \end{pmatrix}
                & ~ = 
                \begin{pmatrix}
                    \dfrac{1+m - h_{K-1}\lambda}{\sqrt{m}} & -1 \\
                    1 & 0
                \end{pmatrix}
                \begin{pmatrix}
                    \Tilde{P}_{(n+1)K-1}(\lambda) \\
                    \Tilde{P}_{(n+1)K-2}(\lambda)
                \end{pmatrix} \\
                & ~ = 
                \begin{pmatrix}
                    \dfrac{1+m - h_{K-1}\lambda}{\sqrt{m}} & -1 \\
                    1 & 0
                \end{pmatrix}
                \cdots
                \begin{pmatrix}
                    \dfrac{1+m - h_{0}\lambda}{\sqrt{m}} & -1 \\
                    1 & 0
                \end{pmatrix}
                \begin{pmatrix}
                    \Tilde{P}_{nK}(\lambda) \\
                    \Tilde{P}_{nK-1}(\lambda)
                \end{pmatrix} \\
                & ~ = A(\lambda)
                \begin{pmatrix}
                    \Tilde{P}_{nK}(\lambda) \\
                    \Tilde{P}_{nK-1}(\lambda)
                \end{pmatrix}\,.
            \end{align}
            
            Looking $K$ steps at a time makes the analysis much easier as the process applying $K$ steps is then homogeneous (we apply $A$ and $A$ doesn't depend on the index of the iterate).
            
            \begin{align}
                \Tilde{P}_{(n+1)K}(\lambda) & ~ = a(\lambda) \Tilde{P}_{nK}(\lambda) + b(\lambda) \Tilde{P}_{nK-1}(\lambda), \\
                \Tilde{P}_{(n+1)K-1}(\lambda) & ~ = c(\lambda) \Tilde{P}_{nK}(\lambda) + d(\lambda) \Tilde{P}_{nK-1}(\lambda).
            \end{align}
            
            Combining the two above equations (First one with incremented $n$ + $b(\lambda)$ times the second one - $d(\lambda)$ times the first one) leads to
            \begin{align}
                \Tilde{P}_{(n+2)K}(\lambda) & ~ = (a(\lambda) + d(\lambda))\Tilde{P}_{(n+1)K}(\lambda) - (a(\lambda)d(\lambda)-b(\lambda)c(\lambda)) \Tilde{P}_{nK}(\lambda) \\
                & ~ = 2\sigma(\lambda;\{h_i\}, m) \Tilde{P}_{(n+1)K}(\lambda) - \Tilde{P}_{nK}(\lambda) \label{eq:tchebyshev_like_k_recursion}
            \end{align}
            where the second inequality is deduced after we recognize
            \begin{equation}
                a(\lambda) + d(\lambda) = \mathrm{Tr}(A(\lambda)) = 2\sigma(\lambda;\{h_i\}, m)
            \end{equation}
            and
            \begin{equation}
                a(\lambda)d(\lambda) - b(\lambda)c(\lambda) = \mathrm{Det}(A(\lambda)) = 1
            \end{equation}
            ($A(\lambda)$ is the product of matrices of determinant 1).
            
            In equation~\eqref{eq:tchebyshev_like_k_recursion} we recognize the recursion verified by e.g. $(T_n(\sigma(\lambda;\{h_i\}, m)))_{n\in\mathbb{N}}$, or $(U_n(\sigma(\lambda;\{h_i\}, m)))_{n\in\mathbb{N}}$, where $T_n$ (resp. $U_n$) denotes the first (resp. second) type Chebyshev polynomial of degree $n$.
            
            Moreover we verify the initialization
            \begin{align}
                \Tilde{P}_0(\lambda) & ~ = 1, \\
                \Tilde{P}_K(\lambda) & ~ = a(\lambda)\Tilde{P}_0(\lambda) + b(\lambda)\Tilde{P}_{-1}(\lambda) \\
                & ~ = a(\lambda) + b(\lambda) \dfrac{m}{1+m}\dfrac{1+m - h_{0}\lambda}{\sqrt{m}}.
            \end{align}
            
            We also notice that
            \begin{equation}
                U_n(\sigma(\lambda;\{h_i\}, m)) + \left( b(\lambda) \dfrac{m}{1+m}\dfrac{1+m - h_{0}\lambda}{\sqrt{m}} - d(\lambda) \right)U_{n-1}(\sigma(\lambda;\{h_i\}, m))
            \end{equation}
            verifies the same recursion of order 2 than $\Tilde{P}_{Kn}$ as well as the same 2 initial terms.
            
            Finally, we conclude
            \begin{equation}
                \Tilde{P}_{nK}(\lambda) = U_n(\sigma(\lambda;\{h_i\}, m)) + \left( b(\lambda) \dfrac{m}{1+m}\dfrac{1+m - h_{0}\lambda}{\sqrt{m}} - d(\lambda) \right)U_{n-1}(\sigma(\lambda;\{h_i\}, m))
            \end{equation}
            
            and
            
            \begin{equation}
                P_{nK}(\lambda) = \left(\sqrt{m}\right)^{nK} \Tilde{P}_{nK}(\lambda)\,.
            \end{equation}
            
            Now we have the full expression of the polynomials associated to algorithm $\ref{algo:cyclical_heavy_ball}$.
            Then we can study it's rate of convergence.
            
            Note for any $r\in\llbracket 0, K-1\rrbracket$, we can have a similar expression of the form
            \begin{equation}
                P_{nK+r}(\lambda) = \left(\sqrt{m}\right)^{nK} \left(Q_r^{1}(\lambda)U_n(\sigma(\lambda;\{h_i\}, m)) + Q_r^{2}(\lambda)U_{n-1}(\sigma(\lambda;\{h_i\}, m)) \right)
            \end{equation}
            with $Q_r^{1}$ and $Q_r^{2}$ some fixed polynomials.
            This is the consequence of the fact that all sequences $\Tilde{P}_{nK+r}(\lambda)$ verify the same recursion formula. Only initialization are different.
            
            In order to study the factor rate of this algorithm, let's first introduce $M$ an upper bound of all the $|Q_r^{i}|$. For instance, let $M$ defined as follow.
            \begin{equation}
                M = \underset{r\in\llbracket0, K-1\rrbracket, i\in\lbrace 1, 2\rbrace}{\max}\underset{\lambda\in\Lambda}{\sup}|Q_r^{i}(\lambda)|.
            \end{equation}
            
            Then,
            \begin{align}
                \|x_{t} - x_*\| & ~ \leq \underset{\lambda \in \Lambda}{\sup} |P_{t}(\lambda)| \|x_0 - x_*\| \\
                & ~ \leq M \left(\sqrt{m}\right)^{t} \left(\underset{\lambda \in \Lambda}{\sup} |U_n(\sigma(\lambda;\{h_i\}, m))| + \underset{\lambda \in \Lambda}{\sup} |U_{n-1}(\sigma(\lambda;\{h_i\}, m))|\right)\|x_0 - x_*\|,
            \end{align}
            with $n=\lfloor\frac{t}{K}\rfloor$.
            
            Set $\sigma_{\sup} \triangleq \underset{\lambda \in \Lambda}{\sup} |\sigma(\lambda;\{h_i\}, m)|$. The worst-case rate verifies
            \begin{align}
                \text{If } \sigma_{\sup} \leq 1, & \text{ then } r_t \leq M \left(\sqrt{m}\right)^{t} \left(n + 1 + n \right) = O\left(t\left(\sqrt{m}\right)^{t}\right). \label{eq:worst_case_rate_analysis_robust_region}\\
                \text{If } \sigma_{\sup} > 1, & \text{ then } r_t = O\left( \left(\sqrt{m}\right)^{t} \left( \sigma_{\sup} + \sqrt{\sigma_{\sup}^2 - 1} \right)^{n}\right).
            \end{align}
            The first case analysis comes from the fact that $U_n$ is bounded by $n+1$ on $[-1, 1]$, while the second cases analysis comes from the fact that $U_n(x)$ grows exponentially fast outside of $[-1, 1]$ at a rate $x+\sqrt{x^2-1}$.
            
            Then the factor rate verifies
            \begin{align}
                \text{If } \sigma_{\sup} \leq 1, & ~ 1-\tau = \sqrt{m}.\\
                \text{If } \sigma_{\sup} > 1, & ~ 1-\tau = \sqrt{m} \left( \sigma_{\sup} + \sqrt{\sigma_{\sup}^2 - 1} \right)^{1/K}.
            \end{align}
            
            It remains to notice that $\sqrt{m} \left( \sigma_{\sup} + \sqrt{\sigma_{\sup}^2 - 1} \right)^{1/K} < 1$ is equivalent to $\sigma_{\sup} < \frac{1+m^k}{2\left(\sqrt{m}\right)^k}$.
            
        \end{proof}
        
        From this factor rate analysis, we can state Proposition~\ref{prop:sequence_stepsizes} of Section~\ref{ssec:resulting_algo}.
        
        \sequencestepsizes*
        
        \begin{proof}
            For now we don't assume assumption~\ref{eq:equation_sequence_stepsizes} yet.
            Set $\sigma_0 \triangleq \sigma(0;\{ h_i\}, m)$.
            Then, by definition~(\ref{eq:sigma_stepsize}) of $\sigma(\lambda;\{ h_i\}, m)$,
            \begin{equation}
                \sigma_0 = \frac{1}{2} \mathrm{Tr} \left(
                \begin{bmatrix}
                    \frac{1+m}{\sqrt{m}}   & -1 \\ 
                    1                           &  0 
                \end{bmatrix}^K
                \right)
                = T_K\left(\frac{1+m}{2\sqrt{m}}\right)
                = \frac{1+m^K}{2\left(\sqrt{m}\right)^K}.
            \end{equation}
            Hence, reversing this equality,
            \begin{equation}\label{eq:sqrt_m_from_sigma_0}
                \sqrt{m} = \left(\sigma_0 - \sqrt{\sigma_0^2-1}\right)^{\frac{1}{K}}.
            \end{equation}
            
            From Theorem~\ref{thm:general_rate_convergence}, we therefore know
            \begin{align}
                \text{If } \sigma_{\sup} \leq 1, & ~ 1-\tau = \left(\sigma_0 - \sqrt{\sigma_0^2-1}\right)^{\frac{1}{K}}.\\
                \text{If } \sigma_{\sup} > 1, & ~ 1-\tau = \left(\sigma_0 - \sqrt{\sigma_0^2-1}\right)^{\frac{1}{K}} \left( \sigma_{\sup} + \sqrt{\sigma_{\sup}^2 - 1} \right)^{1/K}.
            \end{align}
            
            But, one can check that
            
            \begin{equation}
                \left(\sigma_0 - \sqrt{\sigma_0^2-1}\right)^{\frac{1}{K}} \left( \sigma_{\sup} + \sqrt{\sigma_{\sup}^2 - 1} \right)^{1/K} \geq \left(\frac{\sigma_0}{\sigma_{\sup}} - \sqrt{\left(\frac{\sigma_0}{\sigma_{\sup}}\right)^2-1}\right)^{\frac{1}{K}}
            \end{equation}
            
            which shows that a tuning generating the polynomial $\frac{\sigma(\lambda;\{ h_i\}, m)}{\sigma_{\sup}}$ would lead to a better convergence rate.
            Hence, we should look for polynomials $\sigma(\lambda;\{ h_i\}, m)$ verifying $\sigma_{\sup} \leq 1$.
            And then,
            \begin{equation}
                1 - \tau = \sqrt{m} = \left(\sigma_0 - \sqrt{\sigma_0^2-1}\right)^{\frac{1}{K}}.
            \end{equation}
            which explain we aim at maximizing $\sigma_0$ subject to $\sigma_{\sup} \leq 1$ (\eqref{eq:optimal_sigma}).
            
            Finally, we proved \textbf{1)}: if $\sigma(\lambda;\{ h_i\}, m) = \sigma_K^{\Lambda}(\lambda)$, then the tuning is optimal in the sense that this is the one that minimizes the asymptotic rate factor among all $K$ steps-sizes based tuning.
            
            From now, we assume
            \begin{equation}
                \sigma(\lambda;\{ h_i\}, m) = \sigma_K^{\Lambda}(\lambda).
            \end{equation}
            
            Therefore,
            \begin{equation}
                \sigma_0 = \sigma_K^{\Lambda}(0)
            \end{equation}
            and \textbf{2)} is already proven above.
            
            \textbf{3)} follows directly from the definition of $\sigma_K^{\Lambda}(\lambda)$.
            
            Finally, since $\sigma_{\sup} \leq 1$, we know
            \begin{equation}
                1-\tau = \sqrt{m} = \left(\sigma_0 - \sqrt{\sigma_0^2-1}\right)^{1/K}
            \end{equation}
            which proves part of \textbf{4)}.
            
            To prove the expression of the worst-case rate $r_t$, we need to apply the intermediate result~\eqref{eq:worst_case_rate_analysis_robust_region} instead of Theorem~\ref{thm:general_rate_convergence}.

        \end{proof}
    
    \subsection{Example: alternating step-sizes (\texorpdfstring{$K=2$}{K=2})}\label{apx:case-of-a-2-step-sizes-cycle}

        \begin{Prop}\label{prop:2interval_opt_iif_same_size}
            The strategy with 2 step-sizes is optimal on the union of two intervals if and only if they have the same length.
        \end{Prop}
        
        \begin{proof}
            This is a direct consequence of Theorem \ref{thm:optimality_from_equioscillation}, which implies $\sigma_2^{\Lambda}(\muone) = \sigma_2^{\Lambda}(\Ltwo) = 1$ and $\sigma_2^{\Lambda}(\mutwo) = \sigma_2^{\Lambda}(\Lone) = -1$.
            
            This is feasible if and only if $\Ltwo - \mutwo = \Lone - \muone$ since $\sigma_2^{\Lambda}$ is a degree 2 polynomial.
            
            Indeed, set $\sigma_2^{\Lambda}(x) = a(x - b)^2 + c$.
            Then, $\sigma_2^{\Lambda}(\muone) = \sigma_2^{\Lambda}(\Ltwo)$ implies $a(\muone - b)^2 + c = a(\Ltwo - b)^2 + c$, then $|\muone - b| = |\Ltwo - b|$ and finally $b = \frac{\muone + \Ltwo}{2}$.
            Similarly, $\sigma_2^{\Lambda}(\mutwo) = \sigma_2^{\Lambda}(\Lone)$ implies $b = \frac{\mutwo + \Lone}{2}$.
            
            We conclude $\frac{\muone + \Ltwo}{2}=\frac{\mutwo + \Lone}{2}$, and $\Ltwo - \mutwo = \Lone - \muone$.
        \end{proof}
        
        \lambdauniontwointerval*
        
        \begin{proof}
            From Theorem \ref{thm:optimality_from_equioscillation}, 
            \begin{align}
                \sigma_2^{\Lambda}(\muone) & ~ = 1, \\
                \sigma_2^{\Lambda}(\Lone) & ~ = -1, \\
                \sigma_2^{\Lambda}(\mutwo) & ~ = -1, \label{eq:sigma_2_reaches_m1} \\
                \sigma_2^{\Lambda}(\Ltwo) & ~ = 1,
            \end{align}
            and this implies that $\frac{T_n\left( \sigma_2^{\Lambda}(\lambda) \right)}{T_n\left(\sigma_2^{\Lambda}(0)\right)}$ is optimal.
            
            In particular, $\Lone$ and $\mutwo$ are roots of $\sigma_2^{\Lambda}+1$.
            Therefore, we know there exists a constant $c$ such that $\sigma_2^{\Lambda}(\lambda) = c(1 - \frac{\lambda}{\Lone})(1 - \frac{\lambda}{\mutwo}) - 1$.
            Moreover, evaluating this in $\muone$ gives $\sigma_2^{\Lambda}(\muone) = c(1 - \frac{\muone}{\Lone})(1 - \frac{\muone}{\mutwo}) - 1 = 1$, so 
            \begin{align}
                c & ~ = \frac{2}{(1 - \frac{\muone}{\Lone})(1 - \frac{\muone}{\mutwo})} \\
                & ~ = \frac{2\Lone\mutwo}{(\Lone - \muone)(\mutwo - \muone)} \\
                & ~ = 2\frac{\left(\frac{\muone + \Ltwo}{2}\right)^2 - R^2\left(\frac{\Ltwo-\muone}{2}\right)^2}{\frac{1-R^2}{4}(\Ltwo - \muone)^2} \\
                & ~ = 2 \frac{\rho^2 - R^2}{1 - R^2}.
            \end{align}
            
            Then,
            \begin{equation}
                \sigma_2^{\Lambda}(\lambda) = 2 \frac{\rho^2 - R^2}{1 - R^2}(1 - \frac{\lambda}{\Lone})(1 - \frac{\lambda}{\mutwo}) - 1
            \end{equation}
            
            which can be written
            \begin{equation}
                \sigma_2^{\Lambda}(\lambda)= 2\left(\frac{1+m}{2\sqrt{m}}\right)^2\left(1 - \frac{\lambda}{\Lone}\right)\left(1 - \frac{\lambda}{\mutwo}\right) - 1
            \end{equation}
            
            with $\left(\frac{1+m}{2\sqrt{m}}\right)^2 = \frac{\rho^2 - R^2}{1 - R^2}$.
            Finally, $m = \left(\frac{\sqrt{\rho^2 - R^2} - \sqrt{\rho^2 - 1}}{\sqrt{1 - R^2}}\right)^2$.
        \end{proof}
        
        \ratefactoralternatinghb*
        
        \begin{proof}
            From Theorem \ref{thm:general_rate_convergence} applied to $K=2$, we immediately have the above result with
            \begin{equation*}
                \sigma_{\sup} = \underset{\lambda\in\Lambda}{\sup}\left| 2\left(\frac{1+m - \lambda h_0}{2\sqrt{m}}\right)\left(\frac{1+m - \lambda h_1}{2\sqrt{m}}\right) - 1 \right|.
            \end{equation*}
            To conclude the proof, we need to prove that the optimal value of $|\sigma_2^{\Lambda}|$ can only be reached on $\left\{ \muone, \Lone, \mutwo, \Ltwo, (1+m)\frac{h_0 + h_1}{2 h_0 h_1} \right\}$.
            Indeed, $\sigma_2^{\Lambda}$ being convex, its maximal value can only be reached on $\left\{ \muone, \Ltwo \right\}$.
            Its minimal value is reached on $(1+m)\frac{h_0 + h_1}{2 h_0 h_1}$.
            Therefore, over $\Lambda$, the minimal value of $\sigma_2^{\Lambda}$ is reached on $(1+m)\frac{h_0 + h_1}{2 h_0 h_1}$ if the latest belongs to $\Lambda$. Otherwise, its minimal value is reached to the closest point in $\Lambda$ to $(1+m)\frac{h_0 + h_1}{2 h_0 h_1}$, namely, it can be any point of $\left\{ \muone, \Lone, \mutwo, \Ltwo \right\}$.
        \end{proof}
        
        \begin{Prop}[Residual polynomial in the robust region] \label{prop:nice_formulation_for_2_steps_HB_with_optimal_param}
            Assuming $\sigma_2^{\Lambda}(\lambda) \geq -1, \quad \forall\lambda\in\Lambda$,
            the residual polynomial associated with the cyclical heavy ball algorithm is
            \begin{align}
                P_{2n}(\lambda) = m^n & \left[ \frac{2m}{1+m}\,T_{2n}\left(\sqrt{\left(\frac{1+m - \lambda h_0}{2\sqrt{m}}\right)\left(\frac{1+m - \lambda h_1}{2\sqrt{m}}\right)}\right) \right. \nonumber \\
                & + \left. \frac{1-m}{1+m}\,U_{2n}\left(\sqrt{\left(\frac{1+m - \lambda h_0}{2\sqrt{m}}\right)\left(\frac{1+m - \lambda h_1}{2\sqrt{m}}\right)}\right) \right].
            \end{align}
        \end{Prop}
        
        \begin{Rem}
            The assumption $\sigma_2(\lambda) \geq -1, \quad \forall\lambda\in\Lambda$ is verified in the robust region, and is useful here because the term $\left(\frac{1+m - \lambda h_0}{2\sqrt{m}}\right)\left(\frac{1+m - \lambda h_1}{2\sqrt{m}}\right)$ is equal to $\frac{1+\sigma_2(\lambda)}{2}$ and must be positive to make the above expression well defined. Otherwise the result can hold replacing the square root with some complex number, but it  brings no value when we derive the convergence rate from it.
        \end{Rem}
        
        \begin{proof}
            This proof reuses elements of the proof of Theorem \eqref{thm:general_rate_convergence}, especially Equation \eqref{eq:tchebyshev_like_k_recursion}.
            For sake of completeness and simplicity, we prove this result again directly in the special case $K=2$.
            
            We first recall the recursion of Algorithm \ref{algo:cyclical_heavy_ball} for $K=2$.
            For sake of simplicity, we directly projet it onto the eigenspace associated to the eigenvalue $\lambda$ of the Hessian of the objective function.
            \begin{equation}
                \begin{array}{rl}
                    x_{2n+1} - x_* & = (1+m-h_0\lambda)(x_{2n} - x_*) - m (x_{2n-1} - x_*). \\
                    x_{2n+2} - x_* & = (1+m-h_1\lambda)(x_{2n+1} - x_*) - m (x_{2n} - x_*).
                \end{array}
            \end{equation}
            Identifying $x_t - x_* = P_t(\lambda)(x_0 - x_*)$ and $P_t(\lambda) = \left(\sqrt{m}\right)^t\Tilde{P}_t(\lambda)$,
            \begin{equation}
                \begin{array}{rl}
                    \Tilde{P}_{2n+1}(\lambda) & = \frac{1+m-h_0\lambda}{\sqrt{m}}\Tilde{P}_{2n}(\lambda) - \Tilde{P}_{2n-1}(\lambda), \\
                    \Tilde{P}_{2n+2}(\lambda) & = \frac{1+m-h_1\lambda}{\sqrt{m}}\Tilde{P}_{2n+1}(\lambda) - \Tilde{P}_{2n}(\lambda).
                \end{array}
            \end{equation}
            Multiplying the first equation by $\frac{1+m-h_1\lambda}{\sqrt{m}}$ and replacing $\frac{1+m-h_1\lambda}{\sqrt{m}}\Tilde{P}_{2n+1}(\lambda)$ and $\frac{1+m-h_1\lambda}{\sqrt{m}}\Tilde{P}_{2n-1}(\lambda)$ accordingly to the second equation leads to
            \begin{equation}
                \Tilde{P}_{2n+2}(\lambda) + \Tilde{P}_{2n}(\lambda) = \frac{1+m-h_0\lambda}{\sqrt{m}}\frac{1+m-h_1\lambda}{\sqrt{m}}\Tilde{P}_{2n}(\lambda) - \left(\Tilde{P}_{2n}(\lambda) + \Tilde{P}_{2n-2}(\lambda)\right)
            \end{equation}
            which can be written as in equation \eqref{eq:tchebyshev_like_k_recursion}
            \begin{equation}
                \Tilde{P}_{2n+2}(\lambda)= \left(\frac{1+m-h_0\lambda}{\sqrt{m}}\frac{1+m-h_1\lambda}{\sqrt{m}}-2\right)\Tilde{P}_{2n}(\lambda) - \Tilde{P}_{2n-2}(\lambda).
            \end{equation}
            Moreover,
            \begin{equation}
                \begin{array}{rl}
                    x_{1} - x_* & = (1-\frac{h_0}{1+m}\lambda)(x_{0} - x_*), \\
                    x_{2} - x_* & = (1+m-h_1\lambda)(x_{1} - x_*) - m (x_{0} - x_*),
                \end{array}
            \end{equation}
            leading to the initialization
            \begin{equation}
                \begin{array}{rl}
                    \Tilde{P}_{1}(\lambda) & = \frac{1}{\sqrt{m}}(1-\frac{h_0}{1+m}\lambda)\Tilde{P}_{0}(\lambda), \\
                    \Tilde{P}_{2}(\lambda) & = \frac{1+m-h_1\lambda}{\sqrt{m}}\Tilde{P}_{1}(\lambda) - \Tilde{P}_{0}(\lambda).
                \end{array}
            \end{equation}
            hence,
            \begin{equation}
                \Tilde{P}_{2}(\lambda) = \left(\frac{1}{1+m}\frac{1+m-h_0\lambda}{\sqrt{m}}\frac{1+m-h_1\lambda}{\sqrt{m}}-1\right)
            \end{equation}
            and recall
            \begin{equation}
                \Tilde{P}_{0}(\lambda) = 1.
            \end{equation}
            
            It remains to notice that
            
            \begin{align}
                & \frac{2m}{1+m}T_{2n}\left(\sqrt{\left(\frac{1+m - \lambda h_0}{2\sqrt{m}}\right)\left(\frac{1+m - \lambda h_0}{2\sqrt{m}}\right)}\right) \nonumber \\
                + & \frac{1-m}{1+m}U_{2n}\left(\sqrt{\left(\frac{1+m - \lambda h_0}{2\sqrt{m}}\right)\left(\frac{1+m - \lambda h_0}{2\sqrt{m}}\right)}\right)
            \end{align}
            
            verifies the same recursion as well as the same initialization for $n=0$ and $n=1$. This allows us to identify the 2 sequences of polynomials
            \begin{align}
                \Tilde{P}_{2n}(\lambda) & = \frac{2m}{1+m}T_{2n}\left(\sqrt{\left(\frac{1+m - \lambda h_0}{2\sqrt{m}}\right)\left(\frac{1+m - \lambda h_0}{2\sqrt{m}}\right)}\right) \nonumber \\
                & + \frac{1-m}{1+m}U_{2n}\left(\sqrt{\left(\frac{1+m - \lambda h_0}{2\sqrt{m}}\right)\left(\frac{1+m - \lambda h_0}{2\sqrt{m}}\right)}\right)
            \end{align}
            
            which concludes the proof.
            
        \end{proof}

        \rateconvergenceheavyballalternatingstepsize*
        
        \begin{proof}
            From Proposition~\ref{prop:lambda_union_two_interval}, Algorithm \ref{algo:alternating_heavy_ball}'s parameter make $\sigma(\lambda;\{ h_i\}, m) = \sigma_2^{\Lambda}$. In particular, by definition,
            \begin{equation}
                -1 \leq 2\left(\frac{1+m - \lambda h_0}{2\sqrt{m}}\right)\left(\frac{1+m - \lambda h_1}{2\sqrt{m}}\right) - 1 \leq 1.
            \end{equation}
            and then
            \begin{equation}
                0 \leq \sqrt{\left(\frac{1+m - \lambda h_0}{2\sqrt{m}}\right)\left(\frac{1+m - \lambda h_0}{2\sqrt{m}}\right)} \leq 1.
            \end{equation}
            
            And we know that $\forall x \leq 1$,  $T_n(x)\leq1$ and $U_n(x) \leq n+1$.
            
            Therefore, using optimal parameters, and from Proposition~\ref{prop:nice_formulation_for_2_steps_HB_with_optimal_param}
            \begin{equation}
                \Tilde{P}_{2n}(\lambda) \leq \frac{2m}{1+m} + (2n+1)\frac{1-m}{1+m} = 1 + 2n\frac{1-m}{1+m}.
            \end{equation}
            
            And the worst-case rate is then upper bounded
            \begin{equation}
                r_t = \left(1 + t\frac{1-m}{1+m}\right) \left(\sqrt{m}\right)^{t}
            \end{equation}
            for all $t$ even.
            
            It remains to plug $m$ expression into the above to conclude.
        \end{proof}
        
        Note that in the proof above, all the expressions are symmetric in $(h_0, h_1)$, which implies that swapping those 2 step-sizes doesn't impact this statement.
        
        \begin{Rem}
            The previous statement provides the convergence rate of \Cref{algo:alternating_heavy_ball}. It does not state that this is the optimal way to tune \Cref{algo:cyclical_heavy_ball}, but comparing the obtained rate to the one of \Cref{algo:cyclical_chebyshev_2_steps} does. Another way to derive the optimal parameters, is to start from the result of \Cref{thm:rate_factor_alternating_hb} applied on a 2 step-sizes strategy, or from the result of \Cref{prop:nice_formulation_for_2_steps_HB_with_optimal_param}. This leads to minimizing $m$ under the constraints that $\zeta(\lambda) \triangleq \left(\frac{1+m - \lambda h_0}{2\sqrt{m}}\right)\left(\frac{1+m - \lambda h_1}{2\sqrt{m}}\right)$ has values between 0 and 1 on $\Lambda = [\muone,\,\Lone]\cup[\mutwo,\,\Ltwo]$. By symmetry of $\Lambda$ and the convex parabola $\zeta$, we know that optimal parameters verify $\zeta(\Lone) = \zeta(\mutwo) = 0$. And therefore, $\zeta(\muone) = \zeta(\Ltwo) = 1$ maximizes the range of allowed $m$. This way we recover the tuning of \Cref{algo:alternating_heavy_ball}.
            Note that $\zeta$ is related to $\sigma_2^{(\Lambda)}$ through the relation $\sigma_2^{(\Lambda)} = 2\zeta - 1$, and therefore the 4 mentioned equalities are equivalent to the equioscillation property.
        \end{Rem}
        
        The next theorem sums up the results of Proposition~\ref{prop:asymptotic_rate_small_kappa} and Table~\ref{tab:speedups}.
        
        \begin{Th}[Asymptotic speedup of HB with alternating step-sizes]\label{thm:asymptotic_speed_up_of_HB_wit h_alternating_step_sizes}
        ~
            \begin{enumerate}
                \item Let $R\in [0, 1)$ be a fixed number, then $\sqrt{m} \underset{\kappa\rightarrow 0}{=} 1 - \dfrac{2\sqrt{\kappa}}{\sqrt{1-R^2}} + o(\sqrt{\kappa})$.
                \item Let 
                \[
                    R(\kappa) \underset{\kappa\rightarrow 0}{=} 1 - \dfrac{\sqrt{\kappa}}{2} + o(\sqrt{\kappa}), \qquad \text{i.e.,} \quad \Lambda \approx [\mu, \mu + \frac{\sqrt{\mu L}}{4}] \cup [L - \frac{\sqrt{\mu L}}{4}, L],
                \]
                then $\sqrt{m} \underset{\kappa\rightarrow 0}{=} 1 - 2\sqrt[4]{\kappa} + o(\sqrt[4]{\kappa})$, therefore leasing to a new square root acceleration.
                \item Let 
                \[
                    R(\kappa) \underset{\kappa\rightarrow 0}{=} 1 - 2\gamma\kappa + o(\kappa), \qquad \text{i.e.,} \quad \Lambda \approx [\mu, (1+\gamma)\mu] \cup [L - \gamma \mu, L],
                \]
                then $\sqrt{m} \underset{\kappa\rightarrow 0}{=} \sqrt{1 + \dfrac{1}{\gamma}} - \sqrt{\dfrac{1}{\gamma}} + o(\kappa)$, therefore leading to a constant complexity.
            \end{enumerate}
        \end{Th}
        
        This is summed up in the Table~\ref{tab:speedups_apx}.
        
        \begin{table}[ht!]
            {
            \begin{center}
                {\renewcommand{\arraystretch}{1.8}
                \begin{tabular}{@{}llll@{}}
                \specialrule{2pt}{1pt}{1pt}
                Relative gap $R\;\;$ & Set $\Lambda$ & Rate factor $\tau$ & Speedup $\tau/\tau^{\text{PHB}}$ \\
                \specialrule{2pt}{1pt}{1pt}
                $R\in[0, 1)$ & $[\mu, \mu + \frac{1-R}{2}(L-\mu)]\cup[L - \frac{1-R}{2}(L-\mu), L]$ &  $\frac{2\sqrt{\kappa}}{\sqrt{1-R^2}}$& $(1-R^2)^{-\frac{1}{2}}$ \\ \hline
                $R=1 - \sqrt{\kappa}/2$  & $[\mu, \mu + \frac{\sqrt{\mu L}}{4}] \cup [L - \frac{\sqrt{\mu L}}{4}, L]$  & $2\sqrt[4]{\kappa}$ & $\kappa^{-\frac{1}{4}}$  \\ \hline
                $R=1 - 2\gamma\kappa$ & $[\mu, (1+\gamma)\mu] \cup [L - \gamma \mu, L]$ &  indep. of $\kappa$ &  $O(\kappa^{-\frac{1}{2}})$\\ 
                \specialrule{2pt}{1pt}{1pt}
                \end{tabular}
                \caption{Case study of the convergence of Algorithm \ref{algo:alternating_heavy_ball} as a function of $R$, in the regime where $\kappa \rightarrow 0$. The \textbf{first line} corresponds to a situation where $R$ is independent of $\kappa$, and we observe a constant gain w.r.t. heavy ball. The {\bfseries second line} study a setting in which $R$ depends on $\sqrt{\kappa}$, meaning the two intervals in $\Lambda$ are relatively small. The asymptotic rate reads $(1-2\sqrt[4]{\kappa})^t$, beating the $(1-2\sqrt{\kappa})^t$ lower bound.
                Finally, in the \textbf{third line}, $R$ depends on $\kappa$, the two intervals in $\Lambda$ are so small that the convergence becomes $O(1)$, i.e., is independent of $\kappa$.\vspace{-2ex}}
                \label{tab:speedups_apx}}
            \end{center}}
        \end{table}
        
        \begin{proof}
            ~
            \begin{enumerate}
                \item Let $R\in[0, 1)$. The momentum $m$ satisfies
                \begin{align*}
                    \sqrt{m} & \underset{\kappa\rightarrow 0}{=} \dfrac{\sqrt{1+O(\kappa) - R^2} - \sqrt{4\kappa + O(\kappa^2)}}{\sqrt{1-R^2}} \\
                    & \underset{\kappa\rightarrow 0}{=} \dfrac{\sqrt{1 - R^2} + O(\kappa) - 2\sqrt{\kappa} + O(\kappa)}{\sqrt{1-R^2}} \\
                    & \underset{\kappa\rightarrow 0}{=} 1 - \dfrac{2\sqrt{\kappa}}{\sqrt{1-R^2}} + O(\kappa).
                \end{align*}
                
                \item Let $R(\kappa) \underset{\kappa\rightarrow 0}{=} 1 - \dfrac{\sqrt{\kappa}}{2} + o(\sqrt{\kappa})$. The momentum $m$ verifies
                \begin{align*}
                    \sqrt{m} & = \sqrt{\dfrac{\left( \frac{1+\kappa}{1-\kappa} \right)^2 - R^2}{1 - R^2}} - \sqrt{\dfrac{\left( \frac{1+\kappa}{1-\kappa} \right)^2 - 1}{1 - R^2}} \\
                    & = \sqrt{\dfrac{\left( \frac{1+\kappa}{1-\kappa} \right)^2 - 1}{1 - R^2} + 1} - \sqrt{\dfrac{\left( \frac{1+\kappa}{1-\kappa} \right)^2 - 1}{1 - R^2}}.
                \end{align*}
                We first focus on
                \begin{align*}
                    \dfrac{\left( \frac{1+\kappa}{1-\kappa} \right)^2 - 1}{1 - R^2} & \underset{\kappa\rightarrow 0}{=} \frac{4\kappa + O(\kappa^2)}{\sqrt{\kappa} + o(\sqrt{\kappa})} \\
                    & \underset{\kappa\rightarrow 0}{=} 4\sqrt{\kappa} + o(\sqrt{\kappa}).
                \end{align*}
                Then,
                \begin{align*}
                    \sqrt{m} & = \sqrt{\dfrac{\left( \frac{1+\kappa}{1-\kappa} \right)^2 - 1}{1 - R^2} + 1} - \sqrt{\dfrac{\left( \frac{1+\kappa}{1-\kappa} \right)^2 - 1}{1 - R^2}}\\
                    & \underset{\kappa\rightarrow 0}{=} \sqrt{1 + 4\sqrt{\kappa} + o(\sqrt{\kappa})} - \sqrt{4\sqrt{\kappa}+o(\sqrt{\kappa})} \\
                    & \underset{\kappa\rightarrow 0}{=} 1 + 2\sqrt{\kappa} + o(\sqrt{\kappa}) - 2\sqrt[4]{\kappa} + o(\sqrt[4]{\kappa}) \\
                    & \underset{\kappa\rightarrow 0}{=} 1 - 2\sqrt[4]{\kappa} + o(\sqrt[4]{\kappa}).
                \end{align*}
                
                \item Let $R(\kappa) \underset{\kappa\rightarrow 0}{=} 1 - 2\gamma\kappa + o(\kappa)$. The momentum $m$ verifies
                \begin{align*}
                    \sqrt{m} & = \sqrt{\dfrac{\left( \frac{1+\kappa}{1-\kappa} \right)^2 - R^2}{1 - R^2}} - \sqrt{\dfrac{\left( \frac{1+\kappa}{1-\kappa} \right)^2 - 1}{1 - R^2}} \\
                    & = \sqrt{\dfrac{\left( \frac{1+\kappa}{1-\kappa} \right)^2 - 1}{1 - R^2} + 1} - \sqrt{\dfrac{\left( \frac{1+\kappa}{1-\kappa} \right)^2 - 1}{1 - R^2}}.
                \end{align*}
                We first focus on
                \begin{align*}
                    \dfrac{\left( \frac{1+\kappa}{1-\kappa} \right)^2 - 1}{1 - R^2} & \underset{\kappa\rightarrow 0}{=} \frac{4\kappa + O(\kappa^2)}{4\gamma\kappa + o(\kappa)} \underset{\kappa\rightarrow 0}{=} \dfrac{1}{\gamma} + o(\kappa).
                \end{align*}
                Then,
                \begin{align*}
                    \sqrt{m} & = \sqrt{\dfrac{\left( \frac{1+\kappa}{1-\kappa} \right)^2 - 1}{1 - R^2} + 1} - \sqrt{\dfrac{\left( \frac{1+\kappa}{1-\kappa} \right)^2 - 1}{1 - R^2}}\\
                    & \underset{\kappa\rightarrow 0}{=} \sqrt{1 + \dfrac{1}{\gamma} + o(\kappa)} - \sqrt{\dfrac{1}{\gamma} + o(\kappa)} \\
                    & \underset{\kappa\rightarrow 0}{=} \sqrt{1 + \dfrac{1}{\gamma}} - \sqrt{\dfrac{1}{\gamma}} + o(\kappa).
                \end{align*}
            \end{enumerate}
        \end{proof}
        
    \subsection{Example: 3 cycling step-sizes}\label{apx:case-of-a-3-step-sizes-cycle}
    
        \begin{Prop}
            The strategy with 3 step-sizes is optimal on the union of two intervals if and only if they are of the form 
            \[
                \left[\mu, \mu + (L-\mu)\left(\frac{1}{2} - \frac{R}{2} + \frac{1-R^2}{4}\right)\right] \cup \left[L - (L-\mu)\left(\frac{1}{2} - \frac{R}{2} - \frac{1-R^2}{4}\right), L\right],
            \]
            for some $R\in[0, 1]$.
        \end{Prop}
        \begin{proof}
            From Theorem \ref{thm:optimality_from_equioscillation}, we know that $T_n(\sigma_3)$ is optimal for all $n$ if and only if, $\Lambda$ is the union of 3 different intervals that are mapped on $[-1, 1]$. Since, we are looking for $\Lambda$ being the union of 2 intervals, we know 2 of the 3 intervals $\Lambda$ is composed of share an extremity.
            Recall $\Lambda = [\mu_1, L_1]\cup[\mu_2, L_2]$. By symmetry, we can assume without loss of generality that $[\mu_1, L_1]$ is mapped to $[-1, 1]$ twice, and $[\mu_2, L_2]$ once.
            Let's then introduce $x\in(\mu_1, L_1)$ and say:
            \begin{align}
                \sigma_3(\mu_1) & ~ = 1, \\
                \sigma_3(x) & ~ = -1, \\
                \sigma_3(L_1) & ~ = 1, \\
                \sigma_3(\mu_2) & ~ = 1, \\
                \sigma_3(L_2) & ~ = -1.
            \end{align}
            Note we also know that $x$ is a local minima of $\sigma_3$, leading to $\sigma_3'(x) = 0$. We now know 3 roots of $\sigma_3+1$ and 3 roots of $\sigma_3-1$, leading to:
            \begin{align}
                \sigma_3(\lambda) - 1 & ~ = c(\lambda - \mu_1)(\lambda - L_1)(\lambda - \mu_2), \\
                \sigma_3(\lambda) + 1 & ~ = c(\lambda - x)^2(\lambda - L_2),
            \end{align}
            for some non-zero constant $c$.
            Here, we want to remove the dependency in $x$ or $c$. Using the two equalities above,
            \begin{equation}
                (\lambda - x)^2(\lambda - L_2) - (\lambda - \mu_1)(\lambda - L_1)(\lambda - \mu_2) = \frac{2}{c}.
            \end{equation}
            Matching the coefficients of the above polynomial leads to
            \begin{align}
                2x + L_2 & ~ = \mu_1 + L_1 + \mu_2 \label{eq:degree_3_x_relationship_1} \\
                & \text{and} \\
                2xL_2 + x^2 & ~ = \mu_1 L_1 + \mu_1 \mu_2 + L_1 \mu_2. \label{eq:degree_3_x_relationship_2}
            \end{align}
            We plug the expression of $x$ we get from the first equality into the second one,
            \begin{equation}\label{eq:equality_btw_eigs_for_degree_3}
                L_2(\mu_1 + L_1 + \mu_2 - L_2) + \left(\frac{\mu_1 + L_1 + \mu_2 - L_2}{2}\right)^2 = \mu_1 L_1 + \mu_1 \mu_2 + L_1 \mu_2.
            \end{equation}
            
            From here, for simplicity, we define
            \begin{equation}
                r_i \triangleq \frac{L_i - \mu_i}{L_2 - \mu_1}, \qquad \mathrm{for~} i\in\left\{1, 2\right\}.
            \end{equation}
            
            Replacing $L_1$ and $\mu_2$ by their expression using $\mu_1$, $L_2$, $r_1$ and $r_2$ leads to
            \begin{equation}
                r_1 = 2\sqrt{r_2} - r_2. \label{eq:degree_3_r_1_r_2_relationship}
            \end{equation}
            
            The reciprocal holds and we can find $x$ using 
            Equation \eqref{eq:degree_3_x_relationship_1} or \eqref{eq:degree_3_x_relationship_2}.
            Note if Equation \eqref{eq:degree_3_r_1_r_2_relationship} holds, we can directly express $\sigma_3$ as the unique polynomial verifying
            \begin{align}
                \sigma_3(\mu_1) & ~ = 1, \\
                \sigma_3(L_1) & ~ = 1, \\
                \sigma_3(\mu_2) & ~ = 1, \\
                \sigma_3(L_2) & ~ = -1.
            \end{align}
            We can therefore conclude
            \begin{equation}
                \sigma_3(\lambda) = 1 - 2\frac{(\lambda-\mu_1)(\lambda-L_1)(\lambda - \mu_2)}{(L_2-\mu_1)(L_2-L_1)(L_2 - \mu_2)}.
            \end{equation}
            
            From the new notations $r_1, r_2, \mu=\mu_1, L=L_2$, we know $T_n(\sigma_3^{\Lambda})$ is optimal for all $n$ if and only if
            \begin{equation}
                \Lambda = [\mu, \mu+r_1(L-\mu)]\cup[L-r_2(L-\mu), L].
            \end{equation}
            
            Let $R$ be
            \begin{equation}
                R = \frac{\mu_2 - L_1}{L_2 - \mu_1}
            \end{equation}
            as in the 2 step-sizes setting. Here, we have $R = 1 - r_1 - r_2$ and we assume $r_1 = 2\sqrt{r_2} - r_2$. Combining those 2 equalities gives:
            \begin{align}
                r_1 & = \frac{1}{2} - \frac{R}{2} + \frac{1-R^2}{4}, \\
                r_2 & = \frac{1}{2} - \frac{R}{2} - \frac{1-R^2}{4},
            \end{align}
            leading to the desired result, i.e., 
            \[ 
            \Lambda = [\mu, \mu + (L-\mu)(\frac{1}{2} - \frac{R}{2} + \frac{1-R^2}{4})]\cup[L - (L-\mu)(\frac{1}{2} - \frac{R}{2} - \frac{1-R^2}{4}), L].
            \]
    
        \end{proof}
        
        \begin{Th}[Asymptotic speedup of heavy ball when cycling over 3 step-sizes] \label{thm:asymptotic_speed_up_of_HB_with_3_cycling_step_sizes}
            Let $R\in [0, 1)$ be a fixed number, then
            \begin{equation}
                \sqrt{m} \underset{\kappa\rightarrow0}{=} 1 - 2\sqrt{\kappa} \sqrt{\frac{1-R^2/9}{1-R^2}} + o(\sqrt{\kappa}).
            \end{equation}
        \end{Th}
        
        \begin{proof}
            From Equation \eqref{eq:sqrt_m_from_sigma_0},
            \[
            \sqrt{m} = \left(\sigma_3^{(\Lambda)}(0) - \sqrt{\sigma_3^{(\Lambda)}(0)^2-1}\right)^{\frac{1}{3}} \;\; \text{with} \; \;\sigma_3^{(\Lambda)}(0) = {1 + 2\frac{\mu_1 L_1 \mu_2}{(L_2 - \mu_1)(L_2 - L_1)(L_2 - \mu_2)}}.
            \]
            
            Using the previous notations,
            \begin{align}
                \mu & ~ = \mu_1, \\
                L & ~ = L_2, \\
                \kappa & ~ = \frac{\mu}{L}, \\
                r_i & ~ \triangleq \frac{L_i - \mu_i}{L_2 - \mu_1}, \qquad \mathrm{for~} i\in\left\{1, 2\right\},
            \end{align}
            we can write $\sigma_3^{(\Lambda)}$ as
            \begin{align}
                \sigma_3^{(\Lambda)}(0) & ~ = 1 + 2\frac{\mu_1 L_1 \mu_2}{(L_2 - \mu_1)(L_2 - L_1)(L_2 - \mu_2)}, \\
                & ~ = 1 + 2\frac{\kappa (\kappa + r_1(1-\kappa))(1 - r_2(1-\kappa))}{(1-\kappa)^3(1-r_1)r_2}, \\
                & ~ \underset{\kappa \rightarrow 0}{=} 1 + 2\kappa \frac{r_1(1-r_2)}{(1-r_1)r_2,} \\
                & ~ = 1 + 2\kappa \frac{\left(\frac{1}{2} - \frac{R}{2} + \frac{1-R^2}{4}\right)\left(\frac{1}{2} + \frac{R}{2} - \frac{1-R^2}{4}\right)}{\left(\frac{1}{2} + \frac{R}{2} + \frac{1-R^2}{4}\right)\left(\frac{1}{2} - \frac{R}{2} - \frac{1-R^2}{4}\right)}, \\
                & ~ = 1 + 2\kappa \frac{9 - 10R^2 + R^4}{1 - 2R^2 + R^4}, \\
                & ~ = 1 + 2\kappa \frac{\left( 1-R^2 \right)\left( 9-R^2 \right)}{\left( 1-R^2 \right)^2}, \\
                & ~ = 1 + 2\kappa \frac{9-R^2}{1-R^2}.
            \end{align}
            
            Then introducing briefly $\varepsilon \triangleq \kappa \frac{9-R^2}{1-R^2} \underset{\kappa\rightarrow 0}{\rightarrow} 0$,
            \begin{align}
                \sqrt{m} & ~ = \left(\sigma_3^{(\Lambda)}(0) - \sqrt{\sigma_3^{(\Lambda)}(0)^2-1}\right)^{\frac{1}{3}}, \\
                & ~ = \left(1 + 2\varepsilon - \sqrt{1+4\varepsilon+4\varepsilon^2 - 1}\right)^{\frac{1}{3}}, \\
                & ~ \underset{\kappa\rightarrow0}{=} 1 - \frac{2}{3}\sqrt{\varepsilon} + o(\sqrt{\varepsilon}).
            \end{align}
            Plugging $\varepsilon$ expression into the latest gives
            \begin{equation}
                \sqrt{m} \underset{\kappa\rightarrow0}{=} 1 - 2\sqrt{\kappa} \sqrt{\frac{1-R^2/9}{1-R^2}} + o(\sqrt{\kappa}). 
            \end{equation}
        \end{proof}
    
\section{Beyond quadratic objective: local convergence of cycling methods}\label{apx:beyond-quadratic-objective:-local-convergence-of-cycling-methods}

    In this section, we prove a result of local convergence of the cyclical heavy ball method out of quadratic setting.
    We first recall the \Cref{thm:local_convergence_non_quadratic} stated in \Cref{sec:local-convergence}:
    \localconvergencenonquadratic*
    
    \begin{proof}
        For any $k$ multiple of $K$, consider $S_k$ the operator applying $k$ steps of cycling Heavy Ball on the iterates $x_t$ and $x_{t-1}$ (note since $k$ is a multiple of $K$, Algorithm~\ref{algo:cyclical_heavy_ball} consists in repeating the operator $S_k$).
        Namely $S_k$ is an operator on $\mathbb{R}^{2d}$ verifying $S_k((x_t, x_{t-1})) = (x_{t+k}, x_{t+k-1})$.
        This operator is a composition of gradients of $f$ and affine functions, and so it is continuously differentiable.
        
        Applying the mean value theorem along each coordinate of $S_k$, we have that there exists a matrix-valued function $M(v_1, v_2)$ for all $v_1, v_2$ in the domain of $S_k$ such that
        \begin{equation}
            S_k(v_1) - S_k(v_2) = M(v_1, v_2)(v_1 - v_2)\,,
        \end{equation}
        where the $i^{th}$ rows of $M(v_1, v_2)$ is the gradient of the $i^{th}$ output of $S_k$ evaluated at  a vector on the segment between $v_1$ and $v_2$.
        
        \begin{equation}
            M(v_1, v_2) = \begin{pmatrix}
                \nabla (S_k)_1(w_1)^T \\
                \vdots \\
                \nabla (S_k)_i(w_i)^T \\
                \vdots \\
                \nabla (S_k)_{2d}(w_{2d})^T \\
            \end{pmatrix} \text{ where }
            \forall i\in\llbracket 1, 2d \rrbracket,
            \left\{
            \begin{array}{l}
                (S_k)_i \text{ denotes the $i$th coordinate of } S_k. \\
                \\
                w_i \text{ is a point on the segment $[v_1, v_2]$.}
            \end{array}
            \right.
        \end{equation}
        
        By continuity of those gradients, taking $v_1$ and $v_2$ sufficiently close to $(x_*, x_*)$, $M(v_1, v_2)$ can be chosen arbitrarily close to the Jacobian of $S_k$ in $(x_*, x_*)$ denoted by $JS_k^*$.
        
        Since by assumption the algorithm converges on the quadratic form induced by $H$ at the rate $1 - \tau$, we conclude that the spectral radius of $JS_k^*$ is upper bounded by $1-\tau$.
        
        From the previous point, we can find a small enough neighborhood of $(x_*, x_*)$ such that $M(v_1, v_2)$ has a spectral radius arbitrarily close to $1-\tau$, in particular smaller than 1.
        
        Furthermore, it's known for any $\varepsilon>0$, there exists an operator norm $\|.\|$ such that $\|M(v_1, v_2)\| < 1 - \tau + \varepsilon$. (see e.g. \citep[Proposition A.15]{bertsekas1997nonlinear}).
        
        Hence, for any $\varepsilon>0$, there exists a neighborhood $V$ of $(x_*, x_*)$ and an operator norm $\|.\|$ as described above such that $S_k$ is a $(1 - \tau + \varepsilon$)-contraction on $V$ for the norm $\|.\|$.
        
        This leads to convergence to the only fixed point $(x_*, x_*)$ with a convergence rate smaller than any $1 - \tau + \varepsilon$.
        
        Moreover, the first step of the Algorithm~\ref{algo:cyclical_heavy_ball} is continuous with respect to $x_0$. Hence, for any $V\in\mathbb{R}^{2d}$ neighborhood of $(x_*, x_*)$, there exists $W\in\mathbb{R}^d$ a neighborhood of $x_*$, such that
        \begin{equation}
            x_0 \in W \Longrightarrow (x_1, x_0) \in V.
        \end{equation}
        
        Finally, for any $\varepsilon>0$, there exists $W$ a neighborhood of $x_*$ such that the Algorithm~\ref{algo:cyclical_heavy_ball} converges to $x_*$ with a rate smaller than $1 - \tau + \varepsilon$.
        
    \end{proof}

\section{Experimental setup}\label{apx:experiments}

Benchmarks we run using a Google colab public instance with a single CPU. Producing the results of Figure \ref{fig:benchmarks} took 50 minutes with this setup. The code to reproduce this figure is attached with the supplementary material in the jupyter notebook benchmarks.ipynb .

\section{Comparison with \texorpdfstring{\citet{oymak2021super}}{Oymak (2021)}}
\label{apx:comparison_with_Oymak}
The work of~\citet{oymak2021super} also exploits cyclical step-sizes for when the spectral structure of the Hessian contains a gap. This work appeared concurrently to the first version of this manuscript and takes a somewhat different stand for exploiting this particular spectral structure. We summarize the main differences in the table below.


\begin{table}[H]
\begin{center}
{\renewcommand{\arraystretch}{1.8}
\begin{tabular}{@{}p{0.2\textwidth}p{0.39\textwidth}p{0.32\textwidth}@{}}\specialrule{2pt}{1pt}{1pt}
& \citet{oymak2021super} & This work
\\ \specialrule{2pt}{1pt}{1pt}
Algorithm & Gradient descent & Heavy ball-type (optimal algorithm)
\\ \hline
Cycle length $K$ & $K$ is a function of the spectral assumptions & $K$ is a choice
\\ \hline
Structure of the cycle & $\underbrace{\eta_+,\,  \dots,\, \eta_+}_{K-1\text{ times}},\, \eta_-$  \citep[Definition 1]{oymak2021super}& $(h_0, \dots, h_{K-1})$ (See Algorithm~\ref{algo:cyclical_heavy_ball})
\\ \hline
Optimal among  chosen scheme & Not proven / discussed & Yes
\\ \hline
Spectral assumption & Bimodal (2 intervals) & Any number of intervals \\ \hline
Spectral assumption in bimodal case & Rate depends on $\frac{L_2}{\mu_2}$ and $\frac{L_1}{\mu_1}$ & Rate depends on $L_2-\mu_2$ and $L_1-\mu_1$
\\ \hline
Typical application case in bimodal & Very strong assumption on $L_1-\mu_1$ and very weak assumption on $L_2-\mu_2$. & Weak assumption on both $L_2-\mu_2$ and $L_1-\mu_1$, more in line with empirical observations \citep{papyan2018full, ghorbani2019investigation}.
\\ \hline
Spectrum is a single interval & Does not recover original rate & Recovers rate and optimal method
\\ \hline
Convergence beyond quadratic objectives & Yes (with extra assumptions) & Local convergence
\\
\specialrule{2pt}{1pt}{1pt}
\end{tabular}}
\end{center}
\end{table}
\vspace{-1em}
We emphasize the following points.
\begin{itemize}
    \item 
While we provide Theorem \ref{thm:general_rate_convergence} which described the convergence rate of any cycling heavy ball (for any cycle), \citet{oymak2021super} only studied gradient descent method (without momentum) for a particular cycle (for which cycle length is not a parameter, but fixed by the eigenvalues).
\item Moreover, our Theorem \ref{prop:sequence_stepsizes} provides the optimal cycle to use for any choice of a cycle length, while optimality is not discussed  in \citet{oymak2021super} and the cycle uses only two different step-sizes, which is somewhat arbitrary.
\item Furthermore, our work highlights acceleration under assumptions that seem more aligned with empirical observation:  \citet{oymak2021super} shows that $L_1 - \mu_1$ needs to be very small for his strategy to be worthwhile, while this is not really the case in our experiments (see~\Cref{fig:mnist_density})
\item Finally, in the general case in which we cannot assume any gap in the spectrum, we naturally recover the classical optimal method and rate. This is not the case in \citet{oymak2021super} which is suboptimal in this setup.
\end{itemize}

In this work, we focus on quadratic minimization and give some local convergence guarantee beyond quadratics. On the other hand, \citet{oymak2021super} provides guarantee beyond this setup, at the cost of very restrictive assumptions.





\end{document}